\documentclass[11pt]{article}

\usepackage{amssymb}
\usepackage{amsmath}
\usepackage{amsthm}
\usepackage[colorlinks=true,linkcolor=blue,citecolor=blue,urlcolor=blue]{hyperref}
\usepackage{cleveref}
\usepackage{amsfonts}
\usepackage{graphicx}
\usepackage{placeins}
\usepackage{comment}
\usepackage{xfrac}
\usepackage{esint}
\usepackage{enumerate}
\usepackage{tikz}
\usepackage{parskip}
\usepackage{caption}
\usepackage{subcaption}
\usepackage{bbm}
\usepackage{mathtools}
\usepackage[margin=1in]{geometry} 
\numberwithin{equation}{section}

\newcommand{\e}{\varepsilon}
\newcommand{\R}{\mathbb R}

\newcommand{\x}{{\tilde{x}}}
\newcommand{\E}{{\tilde{E}}}
\newcommand{\W}{{\tilde{W}}}
\newcommand{\p}{{\mathbf{p}}}

\newcommand{\beq}{\begin{equation}}
\newcommand{\eeq}{\end{equation}}

\newcommand{\n}{{\bf n}}

\newcommand\restr[2]{{
  \left.\kern-\nulldelimiterspace 
  #1 
  \right|_{#2} 
  }}

\newtheorem{theorem}{Theorem}[section]

\newtheorem{proposition}[theorem]{Proposition}
\newtheorem{lemma}[theorem]{Lemma}
\theoremstyle{definition}
\newtheorem{definition}{Definition}[section]
\newtheorem{remark}[theorem]{Remark}
\numberwithin{figure}{section}
\allowdisplaybreaks

\begin{document}

\title{\bf Polar director structure of SmAP\textsubscript{F} phase of bent-core liquid crystals in thin planar cells with bias electric field}
\author{Alec D. Wendland\thanks{%
Department of Mathematics, The University of Connecticut, Storrs, CT, USA, alec.wendland@uconn.edu}
\and Xiaodong Yan\thanks{%
Department of Mathematics, The University of Connecticut, Storrs, CT, USA, xiaodong.yan@uconn.edu } }
\maketitle

\begin{abstract}
We study the polar director structure in thin planar cells filled with bent-core liquid crystals in the ferroelectric smectic-A phase ($\text{SmAP}_\mathrm{F}$).  We analyze a continuum phenomenological model
proposed by Gornik et. al. in \cite{GCV14} and present
proofs establishing the existence of weak solutions and conditional uniqueness of nonconstant solutions. We further investigate the qualitative properties of nontrivial solutions {\color{black} and examine the} effects of {\color{black} a} {\color{black} bias electric field, surface anchoring, and cell thickness on the polar director configuration}. Our results {\color{black} are consistent} with  previous experimental observations and numerical simulations {\color{black} reported} in the physics literature. {\color{black} {\color{black} In addition}, our {\color{black} analysis reveals new parameter-dependent behaviors supported by our numerical simulations and extends results reported from previous literature}.}
\end{abstract}







\section{Introduction}

\label{sec:intro}
Ferroelectric liquid crystals can largely enhance the response to {\color{black} an applied} electric field and are of {\color{black} interest in} technological {\color{black} settings}. Meyer {\it{et al}.} \cite{Meyetal75} first reported liquid crystals with ferroelectric behavior, where they showed that molecular chirality allows for long range polar order{\color{black} ing} in the tilted smectic-C (SmC) phases.  For a long time, much interest has been focused on liquid crystals formed by rod-like or disc-like molecules where polar structures have been observed in the chiral smectic-C (SmC*) phases.   The ferroelectric ordering in SmC* is considered improper since the ordering of molecular dipoles is induced by a nonpolar order of chiral molecules instead of electrostatic dipole-dipole interactions. On the other hand, the discovery of polarization switching in bent-core liquid crystals \cite{Nioetal96} provided a new route to achieve the macroscopic polar order{\color{black} ing} in the so-called proper ferroelectric materials. The efficient packing of bent-core molecules gives rise to a polar order along the kink direction of the molecules, thus spontaneous polarization in the non tilted {\color{black} s}mectic-A {\color{black}(SmA)} phase is possible. 

The first synthesis of bent-core liquid crystals was reported by Vorl\"ander \cite{Vor29,VA32}. It ha{\color{black} d} not caught a lot of interest until the synthetic work by Matsunaga {\it{et al.}} \cite{KMM1991,MasM1991,MM1993} in the early 1990s when one of the molecules, 1,3-phenylene bis [4-(4-
n-octyloxyphenyliminomethyl) benzoate] opened a new era in liquid crystal science. The physical importance of bent-core molecules, however, was not realized until  Niori {\it{et al}.} \cite{Nioetal96}  discovered the spontaneous polarization of bent-core molecules  in 1996. The interlayer structure discovered in \cite{Nioetal96} is anti-ferroelectric. Since then, a variety of mesophases formed by bent-core mesogens have been discovered and their dielectric and electro-optic properties have drawn a lot of interest due to potential applications \cite{Sek97,Lin97,Pel99,Wal00,RT06,TT06,WRO11}.  While Brand {\it{et al}} \cite{BCP92, BCP98, CPB99} predicted the possible emergence of polar phases of achiral bent-core liquid crystals theoretically, it was only until 2011, the ferroelectric ordering in SmA phase ($\text{SmAP}_\mathrm{F}$) was first experimentally observed by Reddy {\it{et al}} \cite{Red11} (see also \cite{Guoetal11} and \cite{Zhu12} for further evidence of ferroelectric behavior of bent-core molecules). To learn more on physical properties in various phases  formed by bent-core molecules, we refer the readers to \cite{TT06,JLOS18}. 

In contrast to the extensive literature on experiments and numerical simulations by physicists, much less work on  rigorous analysis of bent-core liquid crystals can be found in the  literature. In the smectic-C setting, structure and stability of bent core liquid crystal fibers  were addressed  in \cite{Baietal07, BP12}. A  variational model  for bent-core columnar phases was discussed in \cite{GY15} where existence of minimizers and Gamma-convergence results were obtained. For ferroelectric smectic-A phase of bent-core liquid crystals, several one dimensional continuum models have been proposed by physicists to explain the experimental observations.   Guo {\it{et al}} \cite{Guoetal11} introduced a continuum model where the equilibrium polar director field was obtained by minimization of {\color{black} a} free energy. The bulk energy density {\color{black}in \cite{Guoetal11}} is a sum of elastic and electrostatic contributions: 
$$
f=\underbrace{K_{p1}(\overrightarrow{\nabla}\cdot \p)+\frac{1}{2}K_{p2}(\overrightarrow{\nabla}\cdot \p)^2}_{f_{elastic}}+\underbrace{\frac{P_0^2p_x^2}{2\e\e_0}-E_BP_0p_x}_{f_{{\color{black}electrostatic}}}
$$ and  surface energy is of the form $f_S=W_S(p_x\pm1)^2\bigg|_{x={\color{black}\pm\frac{1}{2}L}}$. Here  $P_0$ is the magnitude of polarization,  $E_B$ is the external DC bias electric field, $\e$ is the
static dielectric constant, $\e_0$ is the permittivity of the free space and $K_{p1},K_{p2}$ are elastic constants related to the splay of polarization. {\color{black} The} unit vector $\p=(p_x(x),p_y(x),0)$ is the polar order parameter. The elastic contribution in their model contains only the splay of polarization and the surface anchoring is non-polar quadratic. 
A similar continuum model using the Oseen-Frank expression for the elastic energy {\color{black} was}  proposed by Gornik, \v Cepi\v c and Vaupoti\v{c} (G\v CV) \cite{GCV14}. {\color{black} Their  elastic energy density contains both the splay term and bend term:
$$
f_{elastic}=\frac{1}{2}K_{pS}(\overrightarrow{\nabla }\cdot \p)^2+\frac{1}{2}K_{pB}({\color{black}\p\times}(\overrightarrow{\nabla}\times \p))^2,
$$
where $K_{pS}$ is the polarization splay elastic constant and $K_{pB}$ is the polarization bend elastic constant, and their surface energy considers only the surface polarization:
$$
f_S=-W_Sp_x\bigg|_{x=0}+W_Sp_x\bigg|_{x=L},
$$ where $W_S$ is the strength of the polar surface anchoring.} 
In her PhD thesis, Leskovar \cite{Les16} studied  the structure and dielectric response of  G\v CV's model with an additional linear splay term 
$
 K_l\frac{\partial p_x}{\partial x}
$
in  the elastic energy density $f_{elastic}$  and an additional  non-polar anchoring term $W_N(\p\cdot\nu)^2$ in the surface density{\color{black} , where} $\nu$ is the interior normal unit vector on the surface. In addition, {\color{black} the} effect of different types of surface anchoring on the structure of polar director was discussed in \cite{Les16}. 
 A molecular model for 
 nonchiral bent-core molecules was developed by  Osipov and Pajak in \cite{OP14} where they considered
only the dispersion and electrostatic dipolar interactions
between central parts of bent-core molecules. In the two dimensional case $\Omega_S=(0,S)^2$, $\Gamma_{SH}=(0,S)\times \{0,S\}$, Garc\'ia-Cervera, Giorgi and Joo considered free energy of the form
$$
E(\p)=\int_{\Omega_S}\left(K|\nabla \p|^2-P_0E_Bp_y\right)d\x+W\int_{\Gamma_{SH}}(\p\cdot\nu)^2
$$ 
and studied the boundary vortex formation in \cite{GGJ20}. Their model corresponds to the Oseen-Frank energy when the elastic constant {\color{black} is equal} to the ben{\color{black} d} co{\color{black} n}stant.  In a separate work, the three authors  \cite{GGJ23} derived a reduced two dimensional model from a three dimensional model via Gamma-{\color{black} convergence}.

The results on the static polar structure  in \cite{GCV14,Les16} are mainly based on numerical simulations and rigorous understanding of the equilibrium polar structure, to the best of our knowledge,  is not known. 
In this paper, we revisit the  continuum phenomenological model for $\text{SmAP}_\mathrm{F}$  phase in thin planar cells proposed in \cite{GCV14}.  We present rigorous proof{\color{black} s} of existence of weak solutions and conditional uniqueness of nonconstant profiles and discuss qualitative properties of equilibrium solutions. The effects of bias electric field, cell thickness and strengths of the surface anchoring on the polar director structure are also discussed. Our results {\color{black} closely} match the numerical simulations in \cite{GCV14,Les16}. {\color{black}In addition, our analysis and supporting numerical simulations reveal new parameter-dependent behavior, extending the results previously found in \cite{GCV14}.}

The paper is organized as follows. Section 2 presents some preliminary {\color{black}  definitions, introduces the free energy model, and states the main results}. Rigorous proofs {\color{black} of results are} presented in section 3, and section 4 is devoted to {\color{black} our} numeri{\color{black} c}al simulation{\color{black} s} {\color{black} and discussion}.

\section{Preliminaries} 

\label{sec:prelim}
 Following \cite{GCV14}, we consider the ferroelectric SmA phase ($\text{SmAP}_\mathrm{F}$) made of bent{\color{black} -}core molecules in  thin planar cells of thickness $L$ (see Figure 1, {\color{black}reproducing Figure 1 from \cite{GCV14}}). Label the average direction of the long molecular ax{\color{black} i}s with the director $\n$ and the direction of the short ax{\color{black} i}s with the polar director $\p$. We assume that the smectic order parameter and the director $\n$ are constants, and only the polar director $\p$ varies across the cell. We also assume that the spatial variation of the polar director is the same across all smectic layers.

The free energy of the liquid crystal inside the cell (\cite{GCV14}) is the sum of the bulk and surface contribution

\begin{equation*}\label{eqn:energy}
F=\int fdV+\int f_s dS,
\end{equation*}
where the bulk free energy density $f$ is of the form
\begin{eqnarray}\label{bulkdensity}
f=\frac{1}{2}K_{pS}{(\overrightarrow{\nabla }\cdot \p)}^2+\frac{1}{2}K_{pB}(\p\times(\overrightarrow{\nabla}\times \p))^2+\frac{P_0^2p_x^2}{2\e\e_0}-E_BP_0p_x,
\end{eqnarray}
where $\p=(p_x,p_y,0)$ is a unit vector {\color{black} representing polarization, $\n~{\color{black}=(0,0,1)}$ is the unit vector representing the average direction of the long axis of the bent-core molecule,} $P_0$ is the magnitude of the polarization, $\e$ is the static dielectric constant, $\e_0$ is the permittivity of free space ($\e_0=8.85\times 10^{-12} \text{F}/\text{m})$, $E_B$ is the external DC bias electric field, {\color{black} and} $K_{pS}$ and $K_{pB}$ are the polarization splay elastic constant and the polarization bend elastic constant respectively. In the bent-core liquid crystals, the bend elastic constant $K_{pB}$ is much lower than the splay elastic constant $K_{pS}$.  The third term in \eqref{bulkdensity} is the self-electrostatic energy and the last term presents coupling of polarization with the external DC bias electric {\color{black}field} applied along the $x-$direction.

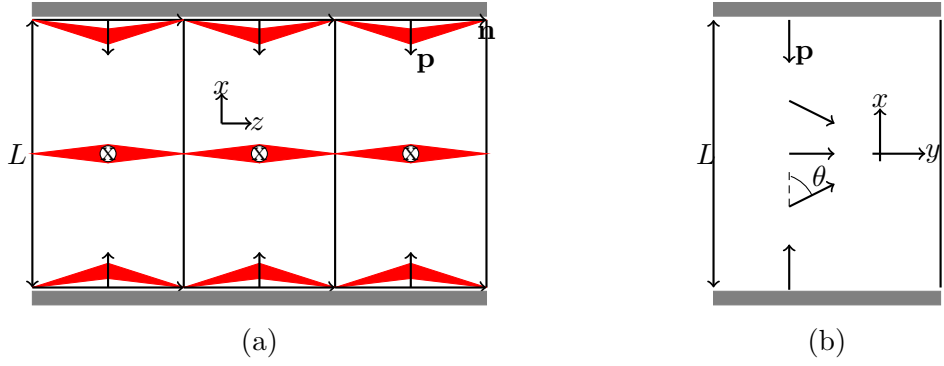
\begin{figure}
\centering
 \begin{tikzpicture}[cross/.style={path picture={ \draw[black] (path picture bounding box.south east)-(path picture bounding box.north west) (path picture bounding box.south west)-(path picture bounding box.north east); }}]
  \draw[gray,fill] (0,3.82) rectangle (6,4);
 \draw[thick,<->] (0,0.23)--(0,3.77);
\draw (-0.2,2) node {$L$};
\draw [thick] (2,0.23)--(2,3.77);
\draw [thick] (4,0.23)--(4,3.77);
\draw [thick] (6,0.23)--(6,3.77);

\draw[thick,->] (0,3.77)--(2,3.77);
\draw[red,fill] (0,3.77)--(1,3.65)--(1,3.45)--cycle;
\draw[red,fill] (2,3.77)--(1,3.65)--(1,3.45)--cycle;
\draw[red,fill] (4,3.77)--(3,3.65)--(3,3.45)--cycle;
\draw[red,fill] (2,3.77)--(3,3.65)--(3,3.45)--cycle;
\draw[red,fill] (4,3.77)--(5,3.65)--(5,3.45)--cycle;
\draw[red,fill] (6,3.77)--(5,3.65)--(5,3.45)--cycle;
\draw[thick,->] (1,3.77)--(1,3.3);
\draw[thick,->] (3,3.77)--(3,3.3);
\draw[thick,->] (5,3.77)--(5,3.3);
 \draw[thick,->] (2,3.77)--(4,3.77);
 \draw[thick,->] (4,3.77)--(6,3.77);
\draw (5.2,3.2) node {$\p$};
\draw (6,3.6) node {$\n$};

\draw[red,fill] (0,2)--(1,2.12)--(1,1.88)--cycle;
\draw[red,fill] (2,2)--(1,2.12)--(1,1.88)--cycle;
\filldraw[fill=white] (1,2) circle (3pt);
\draw (1,2) node {x};
\draw[red,fill] (2,2)--(3,2.12)--(3,1.88)--cycle;
\draw[red,fill] (4,2)--(3,2.12)--(3,1.88)--cycle;
\filldraw[fill=white] (3,2) circle (3pt);
\draw (3,2) node {x};
\draw[red,fill] (4,2)--(5,2.12)--(5,1.88)--cycle;
\draw[red,fill] (6,2)--(5,2.12)--(5,1.88)--cycle;
\filldraw[fill=white] (5,2) circle (3pt);
\draw (5,2) node {x};

\draw[red,fill] (0,0.23)--(1,0.35)--(1,0.55)--cycle;
\draw[red,fill] (2,0.23)--(1,0.35)--(1,0.55)--cycle;
\draw[red,fill] (4,0.23)--(3,0.35)--(3,0.55)--cycle;
\draw[red,fill] (2,0.23)--(3,0.35)--(3,0.55)--cycle;
\draw[red,fill] (4,0.23)--(5,0.35)--(5,0.55)--cycle;
\draw[red,fill] (6,0.23)--(5,0.35)--(5,0.55)--cycle;

 \draw[thick,->] (2.5,2.4) -- (2.9,2.4);
\draw (2.97,2.4) node {$z$};
 \draw[thick,->] (2.5,2.4) -- (2.5,2.8);
 \draw (2.5,2.87) node {$x$};

 \draw[thick,->] (1,0.23)--(1,0.7);
\draw[thick,->] (3,0.23)--(3,0.7);
\draw[thick,->] (5,0.23)--(5,0.7); 
 \draw[thick,->] (2,0.23)--(4,0.23);
\draw[thick,->] (0,0.23)--(2,0.23);
\draw[thick,->] (4,0.23)--(6,0.23);
  \draw[gray,fill] (0,0.18) rectangle (6,0);
\draw (3, -0.5) node {(a)};

\draw[gray,fill] (9,3.82) rectangle (12,4);
 \draw[thick,<->] (9,0.23)--(9,3.77);
\draw  (8.9, 2) node {$L$};
 \draw[thick] (12,0.23)--(12,3.77);
\draw (10.2,3.3) node {$\p$};
\draw[thick,->](10,3.77)--(10,3.2);
\draw[thick,->](10,2.7)--(10.6,2.4);
\draw[thick,->] (10,2)--(10.6,2);
\draw[thick,->] (10,1.3)--(10.6,1.6);
\draw[densely dashed] (10,1.3)--(10,1.8);
\draw (10.3, 1.45) arc (26:70:0.5);
\draw (10.4,1.7) node {$\theta$};
\draw[thick,->] (10,0.2)--(10,0.8);
\draw[thick,->] (11.1,2)--(11.8,2);
\draw (11.9,2) node {$y$};
\draw (11.2,2.7) node {$x $};
\draw[thick,->] (11.2,1.9)--(11.2,2.6); 
\draw[gray,fill] (9,0.18) rectangle (12,0);
\draw (10.5, -0.5) node {(b)};
\end{tikzpicture}
\caption{The cell geometry. $L$ is the cell thickness. $\n$ is the director pointing along the average direction
of the long molecular axes. $\p$ is the polar director pointing to the direction of short axes and local polarization. (a) $\n$ and $\p$ director profiles in the $x-z$ plane. (b) $\p$ director profiles in the $x-y$ plane.}\label{figure}
\end{figure}

We shall consider only the polar surface anchoring. {\color{black} In this way,} the surface energy density $f_S$ is expressed as
\begin{equation*}\label{eqn:surfacedensity}
f_S=-W_Sp_x\bigg|_{x=0}+W_Sp_x\bigg|_{x=L}
\end{equation*}
where $W_S$ is the strength of the polar surface anchoring. 

Introducing dimensionless parameters:
  \begin{equation}\label{eqn:dimpara}
    \tilde{x}=\frac{x}{L},  \hspace{0.1in} \tilde{E}_B=\frac{E_B\e\e_0}{P_0}, \hspace{0.1in} \kappa=\frac{K_{pB}}    {K_{pS}}, \hspace{0.1in}\tilde{W}_S=\frac{W_S\xi}{K_{pS}},
    \end{equation}
with correlation length 
  \begin{equation}\label{eqn:xi}
    \xi=\sqrt{\frac{K_{pS}\e\e_0}{P_0^2}}, 
  \end{equation}
we define the dimensionless free energy 
 \begin{equation}\label{eqn:tilF}
  \tilde{F}=\frac{L}{K_{pS}}F=\int_0^1\tilde{f} d\tilde{x}+\tilde{f}_S,
  \end{equation}
where the dimensionless bulk free energy density $\tilde{f}$ is 
 \begin{eqnarray}\label{eqn:bulkf}
  \tilde{f}=\frac{1}{2}\left(\frac{dp_x}{d\tilde{x}}\right)^2+\frac{1}{2}\kappa\left( \frac{dp_y}{d\tilde{x}}\right)^2+\frac{1}{2}\left(\frac{L}{\xi}\right)^2p_x^2-\left(\frac{L}{\xi}\right)^2\tilde{E}_Bp_x
\end{eqnarray}
and the dimensionless surface free energy density $\tilde{f}_S$ is 
\begin{equation}\label{eqn:surff}
\tilde{f}_S=-\tilde{W}_S\left(\frac{L}{\xi}\right)p_x\bigg|_{\tilde{x}=0}+\tilde{W}_S\left(\frac{L}{\xi}\right)p_x\bigg|_{\tilde{x}=1}.
\end{equation}
{\color{black} Unless indicated otherwise, for the remainder of this paper we assume that {\color{black} $\W_S>0$, $\E_B\ge 0$, $L/\xi\ge 1$, and $\kappa\in (0,1)$.}

Expressing $\p=(\cos \theta,\sin \theta,0)$ where  $\theta({\color{black}\tilde{x}})\in [0,\pi]$ is  the {\color{black}counter}clockwise angle between the ${\color{black}\tilde{x}}-$axis and the polar director $\p${\color{black},} minimization of the free energy over $\theta$ yields the Euler-Lagrange equation
\begin{eqnarray}\label{eqn:el}
(\sin^2\theta+\kappa \cos^2\theta)\theta''+\frac{1}{2}(1-\kappa)\sin2\theta(\theta')^2+\frac{1}{2}\left(\frac{L}{\xi}\right)^2\sin2\theta-\tilde{E}_B\left(\frac{L}{\xi}\right)^2\sin\theta=0
\end{eqnarray}
with boundary conditions 
\begin{equation}\label{eqn:bdry}
 -\tilde{W}_S\left(\frac{L}{\xi}\right)\sin\theta+(\sin^2\theta+\kappa \cos^2\theta)\theta'{\color{black} \bigg\rvert_{\tilde{x}=0,1}}=0.
\end{equation}

\begin{definition}
{\color{black} A function} $\theta\in W^{1,2}(0,1)$ is called a weak solution of \eqref{eqn:el} and \eqref{eqn:bdry} if 
\begin{equation*}
\begin{split}
\int_0^1\left\{-\left(\sin^2\theta)+\kappa \cos^2\theta\right)\theta'\varphi'-\frac{1-\kappa}{2}\sin 2\theta (\theta')^2\varphi+\left(\frac{L}{\xi}\right)^2\sin \theta(\cos \theta-\E_B)\varphi\right\}d\tilde{x}& ~+ \\
\W_S\frac{L}{\xi}\sin\theta\varphi\bigg|_0^1 &=0
\end{split}
\end{equation*} {\color{black} for all $\varphi\in W^{1,2}(0,1)$}.
\end{definition}

\begin{lemma}\label{lm:reg}
If $\theta$ is a weak solution of \eqref{eqn:el} and \eqref{eqn:bdry}, then $\theta \in C^{\infty}(0,1)$.
\end{lemma}

\begin{proof}
{\color{black} The Euler-Lagrange equation} \eqref{eqn:el} can be rewritten as 
\begin{equation}\label{eqn:elw}
\left((\sin^2\theta+\kappa\cos^2\theta)\theta'\right)'=\frac{1-\kappa}{2}\sin 2\theta(\theta')^2-\left(\frac{L}{\xi}\right)^2\sin\theta(\cos\theta-\E_B).
\end{equation}
If $\theta$ is a weak solution of \eqref{eqn:el} and \eqref{eqn:bdry}, then 
$$\left(\sin^2\theta+\kappa \cos^2\theta\right)\theta'\in L^2$$
and 
$$\left(\left(\sin^2\theta+\kappa \cos^2\theta\right)\theta'\right)'=\frac{1-\kappa}{2}\sin 2\theta (\theta')^2-\left(\frac{L}{\xi}\right)^2\sin \theta(\cos \theta-\E_B)\in L^1.$$
{\color{black} The} Sobolev embedding  implies that $\left(\sin^2\theta+\kappa \cos^2\theta\right)\theta'\in L^p(0,1)$ for all $p>1$, which implies that the right hand side of \eqref{eqn:elw} is in $L^q$ for any $q>1$.  Apply{\color{black} ing} elliptic estimates to \eqref{eqn:elw}, we get that
$ \theta\in W^{2,q}(0,1)$ for all $q>1$. The conclusion then follows by a bootstrap argument. 
\end{proof}

Our main results are the following theorems. The first result is an existence {\color{black} and uniqueness} theorem {\color{black} for BVP \eqref{eqn:el}-\eqref{eqn:bdry}}.

\begin{theorem}\label{thm:existence}
 Let $\tilde{E}_B\ge 0, \tilde{W}_S > 0, L/\xi \ge 1,$ and  $\kappa\in (0,1)$. Then there exists at least one solution of \eqref{eqn:el}-\eqref{eqn:bdry}. Moreover, $\theta\equiv 0$ and $\theta\equiv \pi$ are constant solutions of \eqref{eqn:el}-\eqref{eqn:bdry}, and these are the only two constant solutions.
 
 Furthermore, for any fixed constant $C\in \R$, BVP \eqref{eqn:el}-\eqref{eqn:bdry} admits at most one nonconstant solution satisfying
\begin{equation}\label{eqn:idbdry}
\frac{\W_S^2\sin^2\theta(\x)}{\sin^2\theta(\x)+\kappa\cos^2\theta(\x)}-\left(\cos^2\theta(\x)-2\E_B\cos\theta(\x)\right)=C \text{ at } \x=0,1.
\end{equation}
More precisely, letting $g({\color{black}t})=\frac{\W^2_S(1-{\color{black}t}^2)}{1+(\kappa-1){\color{black}t}^2}-({\color{black}t}^2-2{\color{black}t}\E_B)$, we have 
 \begin{enumerate}
  \item If $C$ satisfies one of the following: 
             \begin{itemize}
               \item[a.] $C<-2\tilde{E}_B-1$; 
               \item[b.]$-2\tilde{E}_B-1<C<2\tilde{E}_B-1$; or 
               \item[c.] $C>\max_{|t|\leq 1} g(t)$,
            \end{itemize} then there is no solution of \eqref{eqn:el}-\eqref{eqn:bdry} satisfying \eqref{eqn:idbdry}.
 \item The only two constant solutions of BVP \eqref{eqn:el}-\eqref{eqn:bdry} are $\theta\equiv 0$ and $\theta\equiv \pi$, corresponding to $C=2\tilde{E}_B-1$ and $C=-2\tilde{E}_B-1$, respectively. When $\tilde{E}_B=0$, both solutions correspond to $C=-1$.
\item If there exists a nonconstant solution $\theta$, then $2\tilde{E}_B-1<C<\max_{|t|\leq 1}g(t)$, and this solution is unique.
 \end{enumerate}
\end{theorem}

\begin{remark}\label{rmk:trivsol}
	If there exists a point $x_0\in [0,1]$ such that $\theta(x_0)=0$, then $\cos\theta(x_0)=1$ and \eqref{eqn:idbdry} gives $C=2\E_B-1$. Consequently, part (3) of \Cref{thm:existence} implies that $\theta\equiv 0$, that is, $\cos\theta\equiv 1$ on $[0,1]$. Therefore $\theta\equiv 0$ if and only if $C=2\E_B-1$. Similarly $\theta\equiv \pi$ if and only if $C=-2\E_B-1$. In particular, every nontrivial solution of BVP \eqref{eqn:el}-\eqref{eqn:bdry} satisfies $\theta\in (0,\pi)$.
\end{remark}

{\color{black}
\begin{remark}
Part (3) of Theorem \ref{thm:existence} implies that if $\E_B\rightarrow\infty$, we must have $\W_S\rightarrow \infty$ in order to have a nontrivial solution of \eqref{eqn:el}-\eqref{eqn:bdry}. 
\end{remark}
}

The second theorem addresses the geometry and {\color{black} qualitative} properties of a nontrivial solution. We take $\theta$ to be a nonconstant solution of \eqref{eqn:el}-\eqref{eqn:bdry} and set $p_x(\x)=\cos\theta(\x)$ in the theorems below. 

\begin{theorem}\label{thm:property}
If $\theta$ {\color{black} is any} nonconstant solution of \eqref{eqn:el}-\eqref{eqn:bdry}{\color{black} ,} then
 \begin{enumerate}
  \item $\theta'(\x)>0$ for any $\x\in [0,1].$
  \item If $\E_B=0$, then $\cos \theta(\x)=-\cos\theta(1-\x)$ for any $\x \in [0,1]$. In particular, $\cos\theta(1/2)=0$ and $\cos\theta(\x)\neq 0$ for any $\x\neq \frac{1}{2}$. If $\theta_1(\x)$ and $\theta_2(\x)$ are two nonconstant solutions of \eqref{eqn:el}-\eqref{eqn:bdry} satisfying \eqref{eqn:idbdry} with $C_1>C_2$, then $\cos^2\theta_1(\x)>\cos^2\theta_2(\x)$ for $\x\in [0,1]\setminus\{\frac{1}{2}\}$.
  \item If $\E_B\geq \max\left\{1,\W_S^{2}\frac{\sqrt{\kappa+2+\sqrt{\kappa^2+8\kappa}}}{2\sqrt{2}}\right\}$, then $p_x''<0$ on $[0,1]$.
  \item If $p_x(0)\ge \E_B$ and $p_x(1)\le 0$, {\color{black} then} there exists a unique $x_0\in(0,1)$ such that $\theta''(x_0)=0${\color{black} . Moreover,} $\theta''({\color{black} \tilde{x}})>0$ for ${\color{black} \tilde{x}}\in (x_0,1)$ and $\theta''({\color{black} \tilde{x}})<0$ for ${\color{black} \tilde{x}}\in (0,x_0)$.
\end{enumerate}
\end{theorem}

{\color{black}
The next theorem addresses {\color{black} the response of the endpoint values of $p_x$ to the applied bias field $\E_B$, surface anchoring strength $\W_S$, cell thickness $L/\xi$, and bend-to-splay ratio $\kappa$.}

\allowdisplaybreaks
\begin{theorem}\label{thm:bdry}
We obtain the following compatibility conditions between the boundary values of $p_x$ and {\color{black} the parameters $\E_B$, $\W_S$, $L/\xi$, and $\kappa$.}
\begin{enumerate}
  \item {\color{black} Let $\E_B<1$, $\kappa\in (0,1)$, and $L/\xi\ge 1$}. We have the following necessary conditions on the  surface anchoring strength $\W_S$ to match different values of {\color{black} $p_x$} on the boundary. \begin{enumerate}
  		\item If $p_x(0) > \E_B$ and $p_x(1) < 0$, then $\W_S>\max\{\E_B\frac{\xi}{L},\sqrt{2\E_B-1}\}$ {\color{black} if $\E_B\geq \frac{1}{2}$; and $\W_S>\E_B\frac{\xi}{L}$ if $\E_B\in [0,\frac{1}{2})$},
  		\item If $p_x(0) > \E_B$ and $p_x(1) > 0$, then $\tilde{W}_S < \min\left\{\sqrt{\frac{1-(1-\kappa)\E_B^2}{\kappa}}, \sqrt{\E_B^2 + \frac{1+\kappa}{8}(\frac{\pi\xi}{L})^2}\right\}$.
  		\item If $p_x(0) < \E_B$ and $p_x(1) < 0$, then $\sqrt{\frac{1-(1-\kappa)\E_B^2}{\kappa}}< \W_S< \frac{\pi\xi}{L}\sqrt{\frac{(1+\kappa)-(1-\kappa^2)\E_B^2}{2(1-\E_B^2)}}$.
  		\item If $p_x(0) < \E_B$ and $p_x(1) > 0$, then $\W_S< \min\Bigl\{\sqrt{\E_B^2 + \frac{1+\kappa}{8}\left(\frac{\pi\xi}{L}\right)^2}, \frac{\pi\xi}{L}\sqrt{\frac{1+\kappa}{8}\left(\frac{1-(1-\kappa)\E_B^2}{1-\E_B^2}\right)}\Bigr\}$.
  	\end{enumerate}
  \item {\color{black} Let $\E_B<1$, $\kappa\in (0,1)$, and $L/\xi\ge 1$.  If $\W_S$ satisfies the bound} $$
  \W_S> \max\left\{\sqrt{\E_B^2+ \frac{1+\kappa}{2}\left(\frac{\pi\xi}{L}\right)^2}, \frac{\pi\xi}{L}\sqrt{\frac{1+\kappa}{2}\left(1+\frac{\kappa\E_B^2}{1-\E_B^2}\right)}\right\},
  $$
  then $p_x(0)>\E_B$ and $p_x(1)<0$.
  \item Let $\tilde{W}_S>0$ and $\kappa\in (0,1)$. If $\theta$ is nonconstant and $\tilde{E}_B > \frac{1+\tilde{W}_S^2}{2}$, then $p_x(1)>0$. Moreover, if $\tilde{E}_B>1+\frac{\tilde{W}_S^2}{\kappa}$, then the only solutions of BVP \eqref{eqn:el}-\eqref{eqn:bdry} are constant, and $\theta\equiv 0$ is the unique global minimizer of the energy defined in \eqref{eqn:enI}.
\end{enumerate}
\end{theorem}
}
{\color{black}
\begin{remark}
Part 1 (a) of Theorem \ref{thm:bdry} shows that in order for $p_x(0)>\E_B$ and $p_x(1)<0$, {\color{black} the} surface anchoring needs to be strong enough {\color{black} relative to the applied bias and cell thickness}. Part 2 {\color{black} then} gives a sufficient condition on the surface anchoring to guarantee $p_x(0)>\E_B$ and $p_x(1)<0$. {\color{black} On the other hand,} part 1 (b), (c), and (d) present necessary conditions on the surface anchoring strength for {\color{black} the} other possible boundary conditions on ${\color{black} p_x(\tilde{x})}$. Part 3 shows that for fixed surface anchoring, if the applied bias field is strong enough, then $\theta(1)<\frac{\pi}{2}$, i.e. the polar director {\color{black} reorients in the direction of the applied field at the boundary $\tilde{x}=1$. Part 3 also shows that when the applied bias is very large relative to the anchoring strength and bend-to-splay ratio, the only admissible solutions of \eqref{eqn:el}-\eqref{eqn:bdry} are constant, and $\theta\equiv 0$ is energetically preferable.}
\end{remark}
}

The next  theorem states the {\color{black} asymptotic} effect{\color{black}s} of the bias electric field, {\color{black} surface anchoring strength, and cell thickness on the structure of the polar director.}

\begin{theorem}\label{thm:eff}
{\color{black} Let $\theta$ be any nonconstant solution of \eqref{eqn:el}-\eqref{eqn:bdry}. Then}
\begin{enumerate}
\item If $\E_B<1$ and $\W_S> \max\left\{\sqrt{\E_B^2+ \frac{1+\kappa}{2}\left(\frac{\pi\xi}{L}\right)^2}, \frac{\pi\xi}{L}\sqrt{\frac{1+\kappa}{2}\left(1+\frac{\kappa\E_B^2}{1-\E_B^2}\right)}\right\}$, then $p_x(0)$ increases as $\E_B$ increases. {\color{black} When $\tilde{E}_B=0$,} {\color{black} $\int_0^1 p_x(\tilde{x}) d\tilde{x} =0$}. For energy-minimizing configurations $p_x$, $\int_0^1 p_x(\tilde{x}) d\tilde{x} = 1$ when $\tilde{E}_B\to \infty$.
\item If $\E_B=0$,  then $p_x(0)$ increases from $0$ to $1$ and $p_x(1)$ decreases from $0$ to $-1$ as $\W_S$ increases from 0 to $\infty$. 
\item Fix $\E_B<1$ and $\W_S> \max\left\{\sqrt{\E_B^2+ \frac{1+\kappa}{2}\left(\frac{\pi\xi}{L}\right)^2}, \frac{\pi\xi}{L}\sqrt{\frac{1+\kappa}{2}\left(1+\frac{\kappa\E_B^2}{1-\E_B^2}\right)}\right\}$. If $\frac{L}{\xi}$ increases, then $p_x(0)$ increases {\color{black} and}  $p_x(1)$ decreases.
\end{enumerate}
\end{theorem}

The last theorem states the  local stability of the nontrivial static polar director. 

\begin{theorem}\label{thm:stability}
If $\E_B=0$ and $\W^2_S<\frac{2}{1+\kappa}$ and $\theta$ is {\color{black} a}  nontrivial solution of \eqref{eqn:el}-\eqref{eqn:bdry}, then $Hess_{\theta}(\varphi,\varphi)\geq 0 $ for any $\varphi\in W^{1,2}_0(0,1)$.
\end{theorem}

\section{Main results} \label{sec:lproof}
 In the remaining part of the paper, we shall use $x$ instead of $\x$ for the dimensionless variable. {\color{black} Motivated by \eqref{eqn:tilF}, \eqref{eqn:bulkf}, and \eqref{eqn:surff}, we define the dimensionless energy $I(\theta)$ in terms of the polar director $\theta$ as} 
\begin{eqnarray}\label{eqn:enI}
I(\theta)&=&\frac{1}{2}\int_0^1(\sin^2\theta+\kappa\cos^2\theta)(\theta')^2dx+\frac{1}{2}\left(\frac{L}{\xi}\right)^2\int_0^1(\cos^2\theta-2\cos\theta \E_B )dx\\ \nonumber
&&-\W_S\left(\frac{L}{\xi}\right)\cos\theta\bigg|_{x={\color{black}0}}+ \W_S\left(\frac{L}{\xi}\right)\cos\theta\bigg|_{x=1}.
\end{eqnarray}

\subsection{Existence and uniqueness}
{\color{black} The central result of this section is the proof of \Cref{thm:existence}. We first prove a lemma to establish that any nontrivial solution of BVP \eqref{eqn:el}-\eqref{eqn:bdry} must be strictly monotonically increasing on $(0,1)$.}

\begin{lemma}\label{lm:mono}
If $\theta$ is a smooth nonconstant solution of \eqref{eqn:el}-\eqref{eqn:bdry}, then $\theta'>0$ on $(0,1)$. 
\end{lemma}
\begin{proof}
{\color{black} If $\theta$ is nontrivial, we must have $\pi>\theta(0)>0$. Otherwise the boundary condition \eqref{eqn:bdry} implies $\theta'(0)=0$, so that the Picard-Lindel\"{o}f existence and uniqueness theorem yields $\theta(x)\equiv 0$ or $\theta\equiv \pi$. }Since  $\theta(0)>0$,  {\color{black}\eqref{eqn:bdry}} implies $\theta'(0)>0$. {\color{black} By contradiction,} assume $\theta'(x)$ is no{\color{black} n}positive for {\color{black} some point} $x\in (0,{\color{black}1)}$. {\color{black} In particular,} there exists $x_0\in (0,{\color{black}1})$ such that $\theta'(x_0)=0$. Without loss of generality, we assume  $\theta'(x)>0$ on $[0,x_0)$ and $\theta'(x_0)=0$, thus $\theta''(x_0)\leq 0$. On the other hand, by \eqref{eqn:el}, we have
$$
\theta''(x_0)=\frac{1}{\sin^2\theta(x_0)+\kappa\cos^2\theta(x_0)}\left(\frac{L}{\xi}\right)^2\sin\theta(x_0)(\E_B-\cos\theta(x_0)).
$$
Thus we must have $\E_B\leq\cos\theta(x_0)$, which yields a contradiction if 
 $\E_B>1$. If $\E_B=1$, we conclude that $\theta''(x_0)=0$ and $\cos\theta(x_0)=1.$ Since $\theta(x_0)>\theta(0)>0$, this is impossible. Lastly, if $\E_B<1$, we have $\E_B\leq \cos\theta(x_0)$. Since $\theta''(x_0)\leq 0$ and $\theta'(x_0)=0$, {\color{black}if $\theta''(x_0)=0$, equation \eqref{eqn:el} implies $\theta(x_0)\in \{0,\pi,\arccos\E_B\}$. If $\theta(x_0)=0$ or $\theta(x_0)=\pi$, the Picard-Lindel\"{o}f theorem again implies that  $\theta(x)\equiv 0$ or $\theta\equiv\pi$, a contradiction. If $\cos\theta(x_0)=\E_B$, the Picard-Lindel\"{o}f theorem gives $\cos\theta(x)\equiv \E_B$, which contradicts the boundary condition at $x=0$.} {\color{black}Therefore we must have $\theta''(x_0)<0$ and  $\theta(x_0)<\pi$, so that} $\theta'(x)\leq 0 $ for $x>x_0$ and $x$ near $x_0$. Therefore $\theta(x)<\theta(x_0)$, which yields $\cos\theta(x)~{\color{black}\geq \cos\theta(x_0)}> \E_B$ for  $x>x_0$ and $x-x_0\ll 1$. Plugging back into the equation \eqref{eqn:el}, we conclude $\theta''(x)<0$ for $x>x_0$ and $x-x_0\ll 1$. Thus $\theta'(x)<0$ for $x>x_0$ and $x-x_0\ll 1$. {\color{black}Fix a point $x_1>x_0$, $x_1-x_0\ll 1$. Repeating the argument at $x_1$ using equation \eqref{eqn:el}, we conclude $\theta'(x)<0$ for $x>x_1$ and $x-x_1\ll1$. }Continu{\color{black} ing} this, we get $\theta'(x)<0$ for all  $x>x_0$. In particular{\color{black} ,} $\theta'(1)<0$, a contradiction to the given boundary condition at $x=1$. Therefore we must have $\theta'(x)>0$ on $[0,1]$. 
\end{proof}

{\color{black} We now prove \Cref{thm:existence}. We show existence of a solution by applying Tonelli's direct method in the calculus of variations to the dimensionless energy \eqref{eqn:enI}. We then prove uniqueness by first observing that the Euler-Lagrange equation \eqref{eqn:el} is integrable, leading to an algebraic constraint on the boundary values $\theta(0),\theta(1)$, which we analyze in an argument by cases.}

\begin{proof}[Proof of \Cref{thm:existence}]
Define the admissible class \begin{equation*}
	\mathcal{A} = \{\theta\in H^1(0,1): 0\le \theta\le \pi\}
\end{equation*} and {\color{black} let $(\theta_n)\subset \mathcal{A}$ be a minimizing sequence such that $I(\theta_n)\le C$ for all $n\in \mathbb{N}$. Since $$
I(\theta)\ge \frac{1}{2}\|\theta'\|_{L^2(0,1)}^2-\tilde{E}_B(\frac{L}{\xi})^2-2\tilde{W}_S\frac{L}{\xi}
$$ for all $\theta\in \mathcal{A}$, it follows that $(\theta_n)$ is bounded in $\mathcal{A}$.} Therefore,  we can find $\theta\in \mathcal{A}$ such that, up to a subsequence, $\theta_n \rightharpoonup \theta$ in $W^{1,2}(0,1)$ and $\theta_n \rightarrow \theta$ a.e. on [0,1]. {\color{black} Thus} $\theta \in \mathcal{A}$. Since $I(\theta)$ is lower semi-continuous on $\mathcal{A}$, we have 
$$
I(\theta)\leq \liminf_{n\rightarrow \infty} I(\theta_n).
$$
{\color{black} This shows that } $\theta$ is a minimizer of $I$ in $\mathcal{A}${\color{black} ,} which implies that $\theta$ is a weak solution of \eqref{eqn:el}-\eqref{eqn:bdry}. By Lemma \ref{lm:reg}, $\theta$ is smooth. {\color{black} Multiplying \eqref{eqn:el} by $\theta^\prime$ and integrating, we obtain}
\begin{equation}\label{eqn:fi}
(\sin^2\theta+\kappa\cos^2\theta)(\theta')^2-\left(\frac{L}{\xi}\right)^2(\cos^2\theta-2\cos\theta \E_B)=\tilde{C}
\end{equation}
 for some constant $\tilde{C}.$
Plugging in the boundary data {\color{black} \eqref{eqn:bdry}}, we get for $x=0,1$ 
$$
\frac{\W_S^2\left(\frac{L}{\xi}\right)^2\sin^2\theta(x)}{\sin^2\theta(x)+\kappa\cos^2\theta(x)}-\left(\frac{L}{\xi}\right)^2(\cos^2\theta(x)-2\cos\theta(x) \E_B)=\tilde{C}.
$$
{\color{black} Hence,} $\cos\theta(0)$ and $\cos\theta(1)$ are solutions of 
\begin{equation}\label{eqn:alg}
\frac{\W_S^2(1-{\color{black}t}^2)}{1+(\kappa-1){\color{black}t}^2}-({\color{black}t}^2-2{\color{black}t}\E_B)=\tilde{C}\left(\frac{\xi}{L}\right)^2=C.
\end{equation}
{\color{black} Since} \eqref{eqn:alg} is equivalent to a fourth order algebraic equation, {\color{black} it follows that} it has at most four real solutions. {\color{black} On the other hand,} BVP \eqref{eqn:el}-\eqref{eqn:bdry} admits a solution only if \eqref{eqn:alg} admits {\color{black} a solution} satisfying $|{\color{black}t}|\leq 1$. Set 

\begin{equation}\label{eqn:f}
 f({\color{black}t})=\frac{\W_S^2(1-{\color{black}t^2})}{1+(\kappa-1){\color{black}t}^2}-({\color{black}t}^2-2{\color{black}t}\E_B)-C=g({\color{black}t})-C.
\end{equation}
Since $f({\color{black}t})\rightarrow -\infty$ as ${\color{black}t}\rightarrow \pm \infty$ and 
$$
\lim_{{\color{black}t}\rightarrow -\frac{1}{\sqrt{1-\kappa}}^-}f({\color{black}t})={\color{black} +}\infty, \hspace{0.1 in} \lim_{{\color{black}t}\rightarrow -\frac{1}{\sqrt{1-\kappa}}^+}f({\color{black}t})=-\infty,
$$
and 
$$
\lim_{{\color{black}t}\rightarrow \frac{1}{\sqrt{1-\kappa}}^-}f({\color{black}t})=-\infty, \hspace{0.1 in} \lim_{{\color{black}t}\rightarrow \frac{1}{\sqrt{1-\kappa}}^+}f({\color{black}t})=+\infty,
$$
there exist $-\infty<{\color{black}t}_1< -\frac{1}{\sqrt{1-\kappa}}$ and $\frac{1}{\sqrt{1-\kappa}}<{\color{black}t}_4<\infty$ such that $f({\color{black}t_i})=0$ {\color{black} for $i=1, 4$}. We identify the locations of the other solutions ${\color{black}t}_2,{\color{black}t}_3$ for different values of $C$. For a {\color{black}nontrivial} solution to \eqref{eqn:el}-\eqref{eqn:bdry}, we must have $-1\leq {\color{black}t}_2<{\color{black}t}_3\leq 1$, with $\cos\theta(0)={\color{black}t}_3>0$ and $\cos\theta(1)={\color{black}t}_2<{\color{black}t}_3$ {\color{black} by \Cref{lm:mono}. We proceed by cases on $C$.}

\textbf{Case I: $C<-2\E_B-1.$}
In this case, we have $f(-1)=-(1+2\E_B)-C>0$ and $f(1)=-(1-2\E_B)-C>0$, therefore we must have $$-\frac{1}{\sqrt{1-\kappa}}<{\color{black}t}_2<-1\text{ and }1<{\color{black}t}_3<\frac{1}{\sqrt{1-\kappa}},$$ which implies that \eqref{eqn:el}-\eqref{eqn:bdry} admits no solution. 

\textbf{Case II: $C=-2\E_B-1.$}
In this case, we have $f(-1)=0$ and $f(1)=4\tilde{E}_B>0$. If $\tilde{E}_B>0$, it then follows 
$$
-1={\color{black}t}_2<1<{\color{black}t}_3<\frac{1}{\sqrt{1-\kappa}},
$$
which corresponds to the unique constant solution $\theta(x)\equiv \pi$ for \eqref{eqn:el}-\eqref{eqn:bdry}.

If $\tilde{E}_B=0$, then $C=-1$ and $f(-1)=f(1)=0$, implying that $t_2=-1$ and $t_3=1$. By \Cref{rmk:trivsol}, no nonconstant solution can attain these endpoint values. Therefore, the only solutions are $\theta\equiv 0$ and $\theta\equiv \pi$.

\textbf{Case II{\color{black} I}: $-2\E_B-1<C<2\E_B-1.$}
In this case, we have $f(-1)<0$ and $f(1)>0$. Therefore we have 
$$
-1<{\color{black}t}_2<1<{\color{black}t}_3<\frac{1}{\sqrt{1-\kappa}}.
$$
If $\theta$ is a solution of \eqref{eqn:el}-\eqref{eqn:bdry}, we must have $\cos\theta(0)=\cos\theta(1)={\color{black}t}_2$. By Lemma \ref{lm:mono}, {\color{black} it follows} $\theta(x)\equiv\theta(0)$, which contradicts  the boundary conditions if $\theta(0)  \not\in \{0,\pi\}$. 

\textbf{Case {\color{black} IV}: $C=2\E_B-1${\color{black} .}}
If $\tilde{E}_B=0$, then $C=-1$ and this reduces to Case II. Therefore, assume that $\tilde{E}_B>0$. In this case, we have $f(-1)<0$ and $f(1)=0$. 

If $f'(1)<0$, then $f({\color{black}t})>0$ for ${\color{black}t}<1$ and $1-{\color{black}t}\ll 1$. Thus there exists ${\color{black}t}_2\in (-1,1)$ such that $f({\color{black}t}_2)=0$. {\color{black} This implies $\cos\theta(0)\in \{{\color{black}t}_2,1\}$. If $\cos\theta(0)=1$}, we must have $\theta'(0)=0$ by {\color{black} the} boundary assumptions. By {\color{black} the Picard-Lindel\"{o}f existence and} uniqueness theorem for initial value problem{\color{black} s comprised} of nondegenerate second order ODE,  we conclude {\color{black}  that} $\theta(x)\equiv 0.$ {\color{black} On the other hand, if} $\cos\theta(0)={\color{black}t}_2$, then $\theta(0)\in (0,\pi)$, {\color{black} and so $\theta$ cannot}  be a constant solution due to {\color{black} the} boundary constraints{\color{black} .} Lemma \ref{lm:mono} {\color{black} then implies} $\theta(1)>\theta(0)$, thus $\cos\theta(1)\in  (-1,{\color{black}t}_2)$ is a third solution of $f({\color{black}t})=0$ in $[-1,1]$, a contradiction. If $f'(1)=0$, then the strict concavity of $f$ on $[-1,1]$ implies that $f(t)<0$ for all $t\in [-1,1)$. Thus, the only admissible root is $t=1$, corresponding to $\theta\equiv 0$. If $f'(1)>0$, then there exists ${\color{black}t}_3\in (1,\frac{1}{\sqrt{1-\kappa}})$ such that $f({\color{black}t}_3)=0$. This implies that the only solution to \eqref{eqn:el}-\eqref{eqn:bdry} is $\theta(x)\equiv 0.$

\textbf{Case {\color{black} V}: $2\E_B-1<C<{\color{black}\max_{|{\color{black}t}|\leq 1}g({\color{black}t})}$}.
In this case, we have $f(-1)<0$, $f(1)<0$ and $f({\color{black}t}_0)>0$ for some ${\color{black}t}_0\in (-1,1)$. Thus there exists $-1<{\color{black}t}_2<{\color{black}t}_0<{\color{black}t}_3<1$ such that $f({\color{black}t}_2)=f({\color{black}t}_3)=0$. Since $f$ is strictly concave on $[-1,1]$, these are the only two roots in this interval. In addition, since $f({\color{black}t})\geq f(-{\color{black}t})$ for any ${\color{black}t}>0$, we have ${\color{black}t}_3>0$. 

If $\theta$ is a nonconstant solution of BVP \eqref{eqn:el}-\eqref{eqn:bdry} satisfying \eqref{eqn:idbdry}, then \Cref{lm:mono} implies that $\cos\theta(0)= t_3$ and $\cos\theta(1)= t_2$. Since $f$ has only two roots in $[-1,1]$, the boundary condition \eqref{eqn:bdry} at $x=0$ and uniqueness of the corresponding initial value problem imply that there exists at most one nonconstant solution satisfying \eqref{eqn:idbdry}.

\textbf{Case {\color{black} VI}: $C\geq{\color{black}\max_{|{\color{black}t}|\leq 1}g({\color{black}t})}.$ }
If  $C>{\color{black}\max_{|{\color{black}t}|\leq 1}g({\color{black}t})}$ we have $f({\color{black}t})<0$ for ${\color{black}t}\in [-1,1]$, thus \eqref{eqn:el}-\eqref{eqn:bdry} does not admit a solution.

If $C={\color{black}\max_{|{\color{black}t}|\leq 1}g({\color{black}t})},$ {\color{black} then  $f({\color{black}t})\leq 0$ for ${\color{black}t}\in [-1,1]$ and $f({\color{black}t^{\ast}})=0$ {\color{black}only} if $g({\color{black}t^{\ast}})=\max_{|{\color{black}t}|\leq 1}g({\color{black}t})$.} Therefore, any solution $\theta$ must satisfy \begin{equation*}
	\cos\theta(0)=\cos\theta(1)=t^\ast.
\end{equation*}

Since ${\color{black}g({\color{black}t})\geq g(-{\color{black}t})}$ for ${\color{black}t}>0$, and 
$$
{\color{black}g'({\color{black}t})=-\frac{2\tilde{W}_S^2\kappa {\color{black}t}}{(1+(\kappa-1){\color{black}t}^2)^2}-2{\color{black}t}+2\E_B>0 }
$$
{\color{black}for ${\color{black}t}<0$, we must have ${\color{black}t}^\ast\geq 0$.  Since} 
$${\color{black}g}''({\color{black}t})=-2-\frac{ 2\W^2_S\kappa}{(1+(\kappa-1){\color{black}t}^2)^2}+\frac{8\kappa\W^2_S(\kappa-1){\color{black}t}^2}{(1+(\kappa-1){\color{black}t}^2)^3}<0$$ for {\color{black} all ${\color{black}t} $}, {\color{black}$t^\ast$ is unique.} By \Cref{lm:mono}, no nonconstant solution can satisfy this equality, so the only solutions must be constant. By the boundary condition \eqref{eqn:bdry}, we get $\theta\equiv 0$ or $\theta\equiv \pi$. Since $t^\ast\ge 0$, $\cos\theta\equiv -1$ is impossible, so $\theta \equiv 0$, which occurs precisely when \begin{equation*}
	\max_{|t|\le 1} g(t) = g(1) = 2\tilde{E}_B-1.
\end{equation*} The proof is complete.
\end{proof}

{\color{black} For the remainder of this paper, unless stated otherwise, we will assume that $\W_S,L,\xi,\E_B,\kappa,$ and $C$ are all given such that a nonconstant solution of BVP \eqref{eqn:el}-\eqref{eqn:bdry} exists, and denote this solution by $\theta$.}

\subsection{Admissible endpoint configurations and polar director geometry}
We {\color{black} finish proving} Theorem \ref{thm:property} {\color{black} and Theorem \ref{thm:bdry}} in this section. {\color{black} Unless otherwise specified, $\theta\in C^\infty(0,1)$ will refer to a nonconstant solution of BVP \eqref{eqn:el}-\eqref{eqn:bdry}.}

We first obtain a symmetry result in the case of zero applied bias.

\begin{lemma}\label{lm:sym}
If $\E_B=0$, then $\cos \theta(x)=-\cos\theta(1-x)$ for any $x \in [0,1]$. In particular, $\cos\theta(1/2)=0$ and $\cos\theta(x)\neq 0$ for any $x\neq \frac{1}{2}$. If $\theta_1(x)$ and $\theta_2(x)$ are two nonconstant solutions of \eqref{eqn:el}-\eqref{eqn:bdry} satisfying \eqref{eqn:idbdry} with $C_1>C_2$, then $\cos^2\theta_1(x)>\cos^2\theta_2(x)$ for $x\in [0,1]\setminus\{\frac{1}{2}\}$.
\end{lemma}

\begin{proof}
{\color{black} 
First, when $\E_B=0$, the constraint \eqref{eqn:idbdry} implies that \begin{equation}\label{eqn:idbdry-0E}
	\frac{\W_S^2(1-\cos^2\theta(0))}{1-(1-\kappa)\cos^2\theta(0)}-\cos^2\theta(0) = \frac{\W_S^2(1-\cos^2\theta(1))}{1-(1-\kappa)\cos^2\theta(1)} - \cos^2\theta(1).
\end{equation} Rearranging \eqref{eqn:idbdry-0E}, we find that \begin{equation*}
	\left[1+\frac{\kappa\W_S^2}{(1-(1-\kappa)\cos^2\theta(0))(1-(1-\kappa)\cos^2\theta(1))}\right](\cos^2\theta(1)-\cos^2\theta(0))=0.
\end{equation*} Since $\kappa,\W_S>0$, this implies $\cos^2\theta(0)=\cos^2\theta(1)$. Moreover, since $\theta'>0$ by \Cref{lm:mono}, it follows $\cos\theta(0)=-\cos\theta(1)$.} Plugging {\color{black} this} back into the boundary conditions \eqref{eqn:bdry}, we conclude that $\theta'(0)=\theta'(1)$. On the other hand, take $\psi(x)=\pi-\theta(1-x)${\color{black} . Recalling that $\tilde{E}_B=0$,} then $\psi(x)$ satisfies \eqref{eqn:el} with $\psi(0)=\theta(0)$ and $\psi'(0)=\theta'(1)=\theta'(0)$. By {\color{black} the Picard-Lindel\"of existence and} uniqueness theorem for the  initial value problem of \eqref{eqn:el} {\color{black} together with the initial conditions $\psi(0)=\theta(0)$, $\psi'(0)=\theta'(0)$}, we conclude  that $\theta(x)=\psi(x)$. Therefore $\cos\theta(x)=\cos\psi(x)=-\cos\theta(1-x)$. When $x=\frac{1}{2}$, this implies $\cos\theta(1/2)=0$. Since $\theta'>0$, we conclude that $\cos\theta(x)\neq 0$ for any $x\neq \frac{1}{2}$.

{\color{black} Next,} if $\theta_1(x)$ and $\theta_2(x)$ are two nonconstant solutions of \eqref{eqn:el}-\eqref{eqn:bdry} satisfying \eqref{eqn:idbdry} with $C_1>C_2$, then {\color{black} notice by the above that $\cos\theta_1(\frac{1}{2})= \cos\theta_2(\frac{1}{2})$}. If {\color{black} $\cos\theta_1=\cos\theta_2$} at another point $x_0$, then {\color{black} $\cos\theta_1$ and $\cos\theta_2$} also intersect {\color{black} at} $1-x_0$ by {\color{black} the} anti-symmetry of $\cos\theta_i(x)$. Without loss of generality, we assume $x_0<\frac{1}{2}$ and $\cos\theta_1(x)\neq\cos\theta_2(x)$ for $x\in (x_0,\frac{1}{2})$. Since 
$$
\theta_i'(x)=\frac{L}{\xi}\sqrt{\frac{C_i+\cos^2\theta_i(x)}{1+(\kappa-1)\cos^2\theta_i(x)}},
$$
we must have $\theta_1'(x_0)>\theta_2'(x_0)$ and $\theta_1'(\frac{1}{2})>\theta_2'(\frac{1}{2})$.  On the other hand, {\color{black} $\cos\theta_1(x)-\cos\theta_2(x)$ changes sign at $x=x_0$ and $x=1/2$, {\color{black} from which it follows}} 
$$
\left(\theta_1'(x_0)-\theta_2'(x_0)\right)\cdot \left(\theta_1'\left(\frac{1}{2}\right)-\theta_2'\left(\frac{1}{2}\right)\right)< 0,
$$
{\color{black} a} contradiction. 
\end{proof}

{\color{black} The following two results discuss the second derivative of a nontrivial polar director. We obtain a concavity result in the high bias regime.}

\begin{lemma}\label{lm:snd}
If $\E_B\geq  \max\left\{1,\W_S^{2}\frac{\sqrt{\kappa+2+\sqrt{\kappa^2+8\kappa}}}{2\sqrt{2}}\right\}$ and {\color{black} $p_x(0)>0$}, then $p_x''(x)<0$ for all  $x\in [0,1].$
\end{lemma}
\begin{proof}
We first calculate $$p_x'=-\sin\theta(x)\theta'(x)$$ and $$p_x''=-\sin\theta(x)\theta''(x)-\cos\theta(x)(\theta'(x))^2.$$ Applying equation \eqref{eqn:el}, we get 
\begin{equation}\label{eqn:px}
p_x''=\frac{-\kappa\cos\theta(x)(\theta'(x))^2+\left(\frac{L}{\xi}\right)^2\sin^2\theta(x)(\cos\theta(x)-\E_B)}{\sin^2\theta(x)+\kappa\cos^2\theta(x)}.
\end{equation}
Since $\E_B\geq 1$, we conclude that $p_x''(x)<0$ whenever $\theta(x)\in (0,\frac{\pi}{2})$. {\color{black} By contradiction,} assume there exists $x_1 \in (0,1)$ such that $p_x''(x_1)=0$ and $p_x{\color{black}''}(x)<0$ for all $x<x_1.$ {\color{black} Then} $\theta(x_1)>\frac{\pi}{2}$ and, {\color{black} rearranging \eqref{eqn:px}},
\begin{equation}\label{eqn:x1}
(\theta'(x_1))^2=\frac{1}{\kappa}\left(\frac{L}{\xi}\right)^2\sin^2\theta(x_1)\left(1-\E_B\sec\theta(x_1)\right).
\end{equation}
On the other hand, recall the first integral \eqref{eqn:fi} gives
\begin{equation}\label{eqn:id}
(\sin^2\theta(x)+\kappa \cos^2\theta(x))(\theta'(x))^2-\left(\frac{L}{\xi}\right)^2(\cos^2\theta(x)-2\E_B\cos\theta(x))=\tilde{C}
\end{equation}
{\color{black} for some $\tilde{C}\in \R$.} Apply{\color{black} ing} \eqref{eqn:id} at {\color{black} the} points $x_1,0,1$ and {\color{black} taking} into account the boundary conditions \eqref{eqn:bdry} and \eqref{eqn:x1}, we get 
\begin{eqnarray}\label{eqn:three}
\tilde{C}\left(\frac{\xi}{L}\right)^2&=&\frac{\W^2_S}{(1-\kappa)+\frac{\kappa} {\sin^2\theta(0)}}-(\cos^2\theta(0)-2\E_B\cos\theta(0))\\ \nonumber 
&=&\frac{\W^2_S}{(1-\kappa)+\frac{\kappa} {\sin^2\theta(1)}}-(\cos^2\theta(1)-2\E_B\cos\theta(1))\\ \nonumber 
&=&\frac{1}{\kappa}(1+(\kappa-1)z^2)(1-z^2)\left(1+\frac{\E_B}{z}\right)-(z^2+2\E_Bz), 
\end{eqnarray}
where  $z=-\cos\theta(x_1).$
By Lemma \ref{lm:mono}, we have $\theta(1)>\theta(x_1)>\frac{\pi}{2}${\color{black} .} {\color{black}Since $\E_B\geq 1$, we conclude from \eqref{eqn:three} that }
\begin{eqnarray*}
&&\frac{\W^2_S}{(1-\kappa)+\frac{\kappa} {\sin^2\theta(0)}}\\
&=&\frac{\W^2_S}{(1-\kappa)+\frac{\kappa} {\sin^2\theta(1)}}+\left[\cos\theta(0)-\cos\theta(1)\right]\left[\cos\theta(0)+\cos\theta(1)-2\E_B\right]\\
&<&\frac{\W^2_S}{(1-\kappa)+\frac{\kappa} {\sin^2\theta(1)}}.
\end{eqnarray*}
From this, we conclude that $\sin\theta(0)<\sin\theta(1)$, which implies $\cos\theta(0)>-\cos\theta(1)>-\cos\theta(x_1)=z$. Plugging back into \eqref{eqn:three}, we get {\color{black} the} inequality
\begin{eqnarray}
&&\frac{1}{\kappa}(1+(\kappa-1)z^2)(1-z^2)\left(1+\frac{\E_B}{z}\right) \nonumber\\
&>&z^2+2\E_Bz-\cos^2\theta(0)+2\E_B\cos\theta(0) \nonumber\\
&=&(z+\cos\theta(0))(z-\cos\theta(0)+2\E_B)>2z(\E_B+z), \label{eqn:four}
\end{eqnarray} 
where we used $\E_B \geq 1$ and $0<z<\cos\theta(0)<1$ in the last inequality. Since $\E_B+z>0$, we obtain a quadratic inequality for $y=z^2${\color{black} :}
\begin{equation}\label{eqn:y}
(1+(\kappa-1)y)(1-y)>2\kappa y.
\end{equation}
{\color{black} Inequality} {\color{black} \eqref{eqn:y} can be solved explicitly by 
\begin{equation}\label{eqn:z1}
0<y<y_1=\frac{\kappa+2-\sqrt{\kappa^2+8\kappa}}{2(1-\kappa)},
\end{equation}
which  implies $0<-\cos\theta(x_1)=z<\sqrt{y_1}=z_1<1$.  Since $\E_B\geq 1$, we conclude from equation \eqref{eqn:el} and Lemma \ref{lm:mono} that 
\begin{equation}\label{eqn:inc}
\left[(\sin^2\theta(x)+\kappa\cos^2\theta(x))(\theta'(x))^2\right]'=2\left(\frac{L}{\xi}\right)^2\sin\theta(x)(\E_B-\cos\theta(x)) {\color{black} \theta^\prime(x)} \geq 0
\end{equation}
{\color{black} for all $x\in (0,1)$.}
 Combining \eqref{eqn:bdry}, \eqref{eqn:x1} and  \eqref{eqn:inc}, we have 
$$
\frac{\W^2_S\sin^2\theta(1)}{\sin^2\theta(1)+\kappa\cos^2\theta(1)} \ge \frac{\sin^2\theta(x_1)+\kappa\cos^2\theta(x_1)}{\kappa}\sin^2\theta(x_1)(1-\E_B\sec\theta(x_1)).
$$

{\color{black} Using \eqref{eqn:z1}}, we can bound $\W^2_S$ from below by 
\begin{eqnarray}\label{eqn:Ebd}
\W^2_S&{\color{black} \ge}&\left((1-\kappa)+\frac{\kappa}{\sin^2\theta(1)}\right)\frac{(1-\kappa)\sin^2\theta(x_1)+\kappa}{\kappa}\sin^2\theta(x_1)({1-\E_B\sec\theta(x_1)})\\ \nonumber
&{\color{black} >}& \left((1-\kappa)+\frac{\kappa}{\sin^2\theta(1)}\right)\frac{1+(\kappa-1)z_1^2}{\kappa}(1-z_1^2)\left({1+\frac{\E_B}{z_1}}\right)\\ \nonumber 
&{\color{black} >}&
{\color{black}  \left(\frac{\sqrt{\kappa^2+8\kappa}-\kappa}{2\kappa}\right) 
\left(\frac{\sqrt{\kappa^2+8\kappa}-3\kappa}{2(1-\kappa)}\right)\left(\E_B\sqrt{2}\sqrt{\frac{1-\kappa}{\kappa+2-\sqrt{\kappa^2+8\kappa}}}\right)} \nonumber \\
&=& {\color{black} \E_B\sqrt{2}\sqrt{\frac{\kappa+2-\sqrt{\kappa^2+8\kappa}}{1-\kappa}}, \nonumber
}
\end{eqnarray}
{\color{black} where we have also used the estimate $1-\kappa + \frac{\kappa}{\sin^2\theta(1)}\ge 1$.} The conclusion follows from \eqref{eqn:Ebd}.}
\end{proof}

The following lemma asserts a unique inflection point in the case of high surface anchoring.

\begin{lemma}\label{lm:inflpt}
If $p_x(0)\ge \E_B$ and $p_x(1)\le 0$, {\color{black} then} there exists a unique $x_0\in(0,1)$ such that $\theta''(x_0)=0${\color{black} . Moreover,} $\theta''(x)>0$ for $x\in (x_0,1)$ and $\theta''(x)<0$ for $x\in (0,x_0)$.
\end{lemma}
\begin{proof}
Since $\cos\theta(0)\ge \E_B$ and $\cos\theta(1)\le 0$, \eqref{eqn:el} {\color{black} and \eqref{eqn:bdry} yield} $\theta''(1)\ge 0$ and $\theta''(0)< 0$. On the other hand, the  first integral \eqref{eqn:fi} implies
\begin{equation*}
(\sin^2\theta+\kappa\cos^2\theta)(\theta')^2-\left(\frac{L}{\xi}\right)^2(\cos\theta-\E_B)^2=C\left(\frac{L}{\xi}\right)^2.
\end{equation*}
for some constant $C$.
Substitut{\color{black} ing} back into \eqref{eqn:el} {\color{black} and multiplying by $(\sin^2\theta+ \kappa\cos^2\theta)$}, we get 
$$
0=(\sin^2\theta+\kappa\cos^2\theta)^2\theta''+\left(\frac{L}{\xi}\right)^2\sin\theta\cdot  v(\cos\theta),
$$
where {\color{black} 
$$
v(s)=(1-\kappa)s((s-\E_B)^2+C)+(s-\E_B)(1+(\kappa-1)s^2).
$$ Therefore $\theta''$ share{\color{black} s} the same sign as $-v(\cos\theta)$. We calculate
\begin{equation*}\label{eqn:v'}
v'(s)=(1-\kappa)((s-\E_B)^2+C)+1+(\kappa-1)s^2 > 0 \text{ for } s\in ({\color{black}\cos\theta(1),\cos\theta(0)}),
\end{equation*} and recall from \Cref{lm:mono} that $\theta'>0$.} Since $v(\cos\theta(1))\le v(0)<0<v(\tilde{E}_B)\le v(\cos\theta(0))$, we conclude {\color{black} from the intermediate value theorem} that  there  exists  a unique $x_0\in (0,1)$ such that $v(\cos\theta(x_0))=0${\color{black} . Moreover, it follows by monotonicity that}  $v(\cos\theta(x))>0$ for $x\in (0,x_0)$ and $v(\cos\theta(x))<0$ for $x\in (x_0,1)$. Correspondingly, we have $\theta''(x)<0$ for $x\in(0,x_0)$ and $\theta''(x)>0$ for $x\in (x_0,1).$
\end{proof}

This completes the proof of \Cref{thm:property}.

{\color{black} To further characterize nontrivial solutions of BVP \eqref{eqn:el}-\eqref{eqn:bdry}, we analyze the relationship between {\color{black} the anchoring strength, bias field, cell thickness, and bend-to-splay ratio} and the two endpoint values $p_x(0)=\cos\theta(0)$ and $p_x(1)=\cos\theta(1)$. {\color{black} More precisely,} the following result {\color{black} presents necessary conditions on {\color{black} $\W_S$} for given $\kappa, \frac{L}{\xi}$ and bias $\E_B$ for} different endpoint configurations. The proof relies on comparing the first integral \eqref{eqn:fi} to the nonlinear boundary conditions \eqref{eqn:bdry} to yield an algebraic compatibility condition for the boundary values.}

\begin{lemma}\label{lm:bvnec}
	Let $\E_B<1$, {\color{black} $\kappa\in (0,1)$, and $\frac{L}{\xi}\ge 1$.} We have the following {\color{black} necessary conditions on the surface anchoring strength $\W_S$ depending on  the boundary values $p_x(0),p_x(1)$}: \begin{enumerate}
		\item If $p_x(0) > \E_B$ and $p_x(1) < 0$, then $\W_S>\max\{\E_B\frac{\xi}{L},\sqrt{2\E_B-1}\}$ {\color{black} if $\E_B \geq \frac{1}{2}$ and $\W_S >\E_B\frac{\xi}{L}$ if $0\leq \E_B<\frac{1}{2}$.}
		\item If $p_x(0) > \E_B$ and $p_x(1) > 0$, then $\tilde{W}_S < \min\left\{\sqrt{\frac{1-(1-\kappa)\E_B^2}{\kappa}}, \sqrt{\E_B^2 + \frac{1+\kappa}{8}(\frac{\pi\xi}{L})^2}\right\}$.
		\item If $p_x(0) < \E_B$ and $p_x(1) < 0$, then $\sqrt{\frac{1-(1-\kappa)\E_B^2}{\kappa}}< \W_S< \frac{\pi\xi}{L}\sqrt{\frac{(1+\kappa)-(1-\kappa^2)\E_B^2}{2(1-\E_B^2)}}$.
		\item If $p_x(0) < \E_B$ and $p_x(1) > 0$, then $\W_S< \min\left\{\sqrt{\E_B^2 + \frac{1+\kappa}{8}\left(\frac{\pi\xi}{L}\right)^2}, \frac{\pi\xi}{L}\sqrt{\frac{1+\kappa}{8}\left(\frac{1-(1-\kappa)\E_B^2}{1-\E_B^2}\right)}\right\}$.
	\end{enumerate}
\end{lemma}

\begin{proof}
	{\color{black}Rearranging \eqref{eqn:fi}, we get  \begin{equation*}
			\frac{L}{\xi} = \sqrt{\frac{1-(1-\kappa)\cos^2\theta}{C+\cos^2\theta-2\cos\theta\E_B}}\theta',
		\end{equation*} and then integrating from $x=0$ to $x=1$ gives \begin{equation}\label{eqn:fiint}
			\frac{L}{\xi} = \int_{\theta(0)}^{\theta(1)} \sqrt{\frac{1-(1-\kappa)\cos^2\theta}{C+\cos^2\theta-2\cos\theta\E_B}} d\theta.
		\end{equation} 
	where, {\color{black} recalling the proof of \Cref{thm:existence},} $\cos\theta(0)$ and $\cos\theta(1)$ are solutions of \begin{equation}\label{eqn:alg2}
			g(t)=\W_S^2\left(\frac{1-t^2}{1-(1-\kappa)t^2}\right) - t^2+2t\E_B= C.
		\end{equation} 
		
	 We calculate \begin{equation*}
		g'(t) = 2\left[-\kappa\W_S^2\frac{{\color{black}t}}{(1-(1-\kappa)t^2)^2}- ({\color{black}t}-\E_B)\right]
	\end{equation*} and \begin{equation*}
		g''({\color{black}t})= 2\left[-\kappa\W_S^2\left(\frac{1+3(1-\kappa){\color{black}t}^2}{(1-(1-\kappa)t^2)^3}\right)-1\right].
	\end{equation*} Thus $g''({\color{black}t})<0$ for all $t\in [-1,1]$. Noting that $g'({\color{black}t})>0$ on $(-1,0)$ and $g'({\color{black}t})<0$ on $(\E_B,1)$, it follows that $g$ has a unique local maximum point at ${\color{black}t}^\ast\in [0,\E_B]$ and \begin{equation}\label{eqn:minF}
		\min_{t\in [0,\E_B]} g(t) = \min\{g(0),g(\E_B)\}= \min\left\{\W_S^2, \W_S^2\left(\frac{1-\E_B^2}{1-(1-\kappa)\E_B^2}\right)+\E_B^2\right\}.
	\end{equation} We proceed by cases of different boundary values.}
	
	\textbf{Case I:} {\it$p_x(0)>\E_B$ and $p_x(1)<0$. }To satisfy $g(p_x(0))= g(p_x(1))=C$ and \eqref{eqn:minF}, we must have \begin{equation}\label{eqn:bvc1}
		C<{\color{black} \min\left\{\W_S^2, \W_S^2\left(\frac{1-\E_B^2}{1-(1-\kappa)\E_B^2}\right)+\E_B^2\right\}}.
	\end{equation} By the change of variables $y=\cos\theta$ in \eqref{eqn:fiint}, we find that \begin{equation}\label{eqn:Ws1c1}
		\frac{L}{\xi} = \int_{\cos\theta(0)}^{\cos\theta(1)}- \sqrt{\frac{1+\frac{\kappa y^2}{1-y^2}}{C+\cos^2\theta-2\cos\theta\E_B}} dy > \int_0^{\E_B} \sqrt{\frac{1}{\W_S^2}} dy= \frac{\E_B}{\W_S},
	\end{equation} where we have used {\color{black} the assumptions $p_x(0)>\E_B$,  $p_x(1)<0$ and the estimates $\cos^2\theta-2\cos\theta\E_B<0$ for $\cos\theta\in(0,\E_B)$ and  $C< \W_S^2$ }in the last inequality. On the other hand, since $C<{\color{black} g}(0)$ and ${\color{black} g}$ is decreasing on $(p_x(0),1)$, it follows that \begin{equation*}
		{\color{black}-1+2\E_B=g(1)<g(p_x(0))=C<g(0) = \W_S^2}
	\end{equation*} that is, $\W_S^2 > 2\E_B-1$. This and \eqref{eqn:Ws1c1} together prove assertion (1).
	
	\textbf{Case II:} {\it$p_x(0)>\E_B$ and $p_x(1)>0$.} In this case, we must have \begin{equation*}
		g(0) < C < g(\E_B),
	\end{equation*} that is, \begin{equation}\label{eqn:bvc2}
		\W_S^2< C < \W_S^2\left(\frac{1-\E_B^2}{1-(1-\kappa)\E_B^2}\right)+\E_B^2.
	\end{equation} Solving this for $\W_S$ shows \begin{equation}\label{eqn:Ws1c2}
		\W_S^2 < \frac{1-(1-\kappa)\E_B^2}{\kappa}.
	\end{equation} Next, assume for the moment that $C>{\color{black}\E_B^2}$. Since $\theta(0),\theta(1)\in (0,\pi/2)$ and $C>g(0)$, we see by \eqref{eqn:fiint} \begin{align*}
		\frac{L}{\xi} &= \int_{\theta(0)}^{\theta(1)} \sqrt{\frac{1-(1-\kappa)\cos^2\theta}{C+\cos^2\theta-2\cos\theta\E_B}} d\theta \\
			&< \int_0^{\frac{\pi}{2}} \sqrt{\frac{1-(1-\kappa)\cos^2\theta}{{\color{black}C-\E_B^2}}} d\theta \\
		&\le \pi \sqrt{\frac{1+\kappa}{8{\color{black}(C-\E_B^2)}}},
	\end{align*} where we have used H\"older's inequality to obtain the last inequality. It follows \begin{equation}\label{eqn:ubCpx1}
		C < \frac{1+\kappa}{8}\left(\frac{\pi\xi}{L}\right)^2 {\color{black} +\, \E_B^2}.
	\end{equation} Inserting this estimate into \eqref{eqn:bvc2} shows that \begin{equation}\label{eqn:Ws2c2}
		\W_S^2 < \E_B^2 + \frac{1+\kappa}{8}\left(\frac{\pi\xi}{L}\right)^2.
	\end{equation} In the case that $C\le {\color{black}\E_B^2}$,  \eqref{eqn:ubCpx1} is still an upper bound for $C$. Hence, \eqref{eqn:Ws1c2} and \eqref{eqn:Ws2c2} prove assertion (2).
	
	\textbf{Case III:} {\it $p_x(0)<\E_B$ and $p_x(1)<0$.} In this case, it follows that {\color{black}$g(\E_B)<C<g(0)$}, so that \begin{equation}\label{eqn:bvc3}
		{\color{black}\frac{\W_S^2(1-\E_B^2)}{1-(1-\kappa)\E_B^2}+\E_B^2< C < \W_S^2.}
	\end{equation} Solving for $\W_S$ in this case shows \begin{equation}\label{eqn:Ws1c3}
		\W_S^2 > \frac{1-(1-\kappa)\E_B^2}{\kappa}.
	\end{equation} Since {\color{black}$C>g(\E_B)>\E_B^2$}, we estimate using \eqref{eqn:fiint} \begin{align*}
		\frac{L}{\xi} &< \int_0^\pi \sqrt{\frac{1-(1-\kappa)\cos^2\theta}{C{\color{black}-\E_B^2}}} d\theta = \pi \sqrt{\frac{1+\kappa}{2(\color{black}C-\E_B^2)}}.
	\end{align*} Solving this for $C$ shows that \begin{equation}\label{eqn:Cub}
		C < \frac{1+\kappa}{2}\left(\frac{\pi\xi}{L}\right)^2{\color{black}+\,\E_B^2}.
	\end{equation} Then this upper bound together with \eqref{eqn:bvc3} implies \begin{equation}\label{eqn:Ws2c3}
		\W_S^2 < \left(\frac{\pi\xi}{L}\right)^2\left(\frac{1+\kappa}{2}\right)\frac{1-(1-\kappa)\E_B^2}{1-\E_B^2}.
	\end{equation} Combining \eqref{eqn:Ws1c3} with \eqref{eqn:Ws2c3} proves assertion (3).
	
	\textbf{Case IV}: {\it$p_x(0)< \E_B$ and $p_x(1)>0$.} Lastly, if $p_x(0) < \E_B$ and $p_x(1) > 0$, then it follows that $C>\max\{g(0),g(\E_B)\}$. Since $\theta(0),\theta(1)\in (0,\pi/2)$, the estimate \eqref{eqn:ubCpx1} still holds. Thus we have \begin{equation*}
		\frac{1+\kappa}{8}\left(\frac{\pi\xi}{L}\right)^2 > \max\left\{\W_S^2-\E_B^2, \W_S^2\left(\frac{1-\E_B^2}{1-(1-\kappa)\E_B^2}\right)\right\}.
	\end{equation*} Solving this estimate for $\W_S$ proves assertion (4). The proof is complete.
\end{proof}

\begin{remark}
	We comment here that in the proof of \Cref{lm:bvnec}, the condition \begin{equation*}
		C< \W_S^2\left(\frac{1-\E_B^2}{1-(1-\kappa)\E_B^2}\right)
	\end{equation*} is unused in case I because it leads to the estimate \begin{equation*}
		\frac{L}{\xi} > \int_0^{\E_B} \sqrt{\frac{1}{\W_S^2\left(\frac{1-\E_B^2}{1-(1-\kappa)\E_B^2}\right)+\E_B^2}}dy = \frac{\E_B}{\sqrt{{\W_S^2\left(\frac{1-\E_B^2}{1-(1-\kappa)\E_B^2}\right)+\E_B^2}}}.
	\end{equation*} This is equivalent to \begin{equation*}
		\E_B^2+ \W_S^2\left(\frac{1-\E_B^2}{1-(1-\kappa)\E_B^2}\right)> \E_B^2\left(\frac{\xi}{L}\right)^2,
	\end{equation*} which always holds because $L/\xi\ge 1$ by assumption in this paper.
\end{remark}

{\color{black} 
Of the possible endpoint regimes considered in \Cref{lm:bvnec}, the configuration $p_x(0)>\E_B$, $p_x(1)<0$ and the two involving $p_x(1)>0$ are of particular interest because they coincide with the preferred structure of the polar director in the case of high anchoring strength and high bias, respectively. Heuristically, this is expected from the form of the energy $I(\theta)$ \eqref{eqn:enI} from which BVP \eqref{eqn:el}-\eqref{eqn:bdry} is derived. Indeed, when $\W_S$ is large relative to $L/\xi$ and $\E_B$, $\theta$ will favor the endpoint values $\cos\theta(0)\approx 1$, $\cos\theta(1)\approx -1$ in order to minimize the surface contribution to the free energy. On the other hand, when $\E_B$ dominates $\W_S$ and $L/\xi$, the minimizing profile prefers $\cos\theta\approx 1$ in bulk in order to minimize the electrostatic contribution. In particular, at the boundary $x=1$, we expect a competition between the surface preference $p_x(1)\approx -1$ and the electrostatic preference $p_x(1)\approx 1$ driven primarily by the parameters $\W_S$, $\E_B$, and $L/\xi$.}

The following two lemmas provide sufficient conditions for these endpoint configurations.

\begin{lemma}\label{lm:bvbigWs}
	If $\E_B<1$ and $\W_S> \max\left\{\sqrt{\E_B^2+ \frac{1+\kappa}{2}\left(\frac{\pi\xi}{L}\right)^2}, \frac{\pi\xi}{L}\sqrt{\frac{1+\kappa}{2}\left(1+\frac{\kappa\E_B^2}{1-\E_B^2}\right)}\right\}$, then $p_x(0)>\E_B$ and $p_x(1)<0$.
\end{lemma}

\begin{proof}
	We adopt the notation of \Cref{lm:bvnec}; specifically, recall that $p_x(0)$ and $p_x(1)$ are two solutions to $g(t)=C$ on the interval $[-1,1]$. Since $g''(t)<0$ on $[-1,1]$, it follows that these are the only two solutions. Moreover, since $g$ is monotonically increasing on $(-1,0)$ and monotonically decreasing on $(\E_B,1)$, in view of \eqref{eqn:minF} it suffices to show that \begin{equation}\label{eqn:bigWsC}
		{\color{black}C < \min\left\{\W_S^2, \W_S^2\left(\frac{1-\E_B^2}{1-(1-\kappa)\E_B^2}\right)+\E_B^2\right\}.}
	\end{equation} 
	
	To this end, recall from \eqref{eqn:Cub} that \begin{equation*}
		C < \frac{1+\kappa}{2}\left(\frac{\pi\xi}{L}\right)^2{\color{black}+\,\E_B^2}.
	\end{equation*} By the assumption that $\W_S>\sqrt{\E_B^2+\frac{1+\kappa}{2}\left(\frac{\pi\xi}{L}\right)^2}$, we see that \begin{align*}
		\W_S^2> \frac{1+\kappa}{2}\left(\frac{\pi\xi}{L}\right)^2 +\E_B^2 > C.
	\end{align*} Similarly, since $\W_S> \frac{\pi\xi}{L}\sqrt{\frac{1+\kappa}{2}\left(1+\frac{\kappa\E_B^2}{1-\E_B^2}\right)}$, {\color{black}we calculate \begin{align*}
		\W_S^2\left(\frac{1-\E_B^2}{1-(1-\kappa)\E_B^2}\right) +\E_B^2&> \left(\frac{\pi\xi}{L}\right)^2 \left(\frac{1+\kappa}{2}\right)\left(1+\frac{\kappa\E_B^2}{1-\E_B^2}\right)\left(\frac{1-\E_B^2}{1-(1-\kappa)\E_B^2}\right)  \\ 
			&= \frac{1+\kappa}{2}\left(\frac{\pi\xi}{L}\right)^2+\E_B^2 \\
			&> C.
	\end{align*} }This proves \eqref{eqn:bigWsC}. The proof is complete.
\end{proof}

\begin{remark}
	\Cref{lm:bvbigWs} holds even in the case of zero bias because the assumption on the surface anchoring $\W_S$ still enforces \eqref{eqn:bigWsC} when $\E_B=0$. Alternatively, the same conclusion can be reached by combining \Cref{lm:sym,lm:mono}. Indeed, when $\E_B=0$, \Cref{lm:sym} implies $p_x(1/2)=0$ and $p_x(x)\ne 0$ for all $x\ne 1/2$; then, \Cref{lm:mono} gives $p_x(0)> \E_B=0 > p_x(1)$. 
\end{remark}

\begin{lemma}\label{lm:theta1}
	If $\E_B > \frac{1+\W_S^2}{2}$, then $p_x(1)>0$. {\color{black} Moreover, if $\E_B>1+\frac{\W_S^2}{\kappa}$, then $p_x(x)\equiv \pm 1$}, that is, $\theta\equiv 0$ or $\theta\equiv \pi$, and $\theta\equiv 0$ is a global minimizer of the energy \eqref{eqn:enI}.
\end{lemma}

\begin{proof}
	To prove the first part, recalling from the proof of \Cref{lm:bvnec} that $g(p_x(0))=g(p_x(1))=C$ and that $g$ is monotonically increasing on $(-1,0)$, it suffices to show that $C>g(0)=\W_S^2$. In view of \Cref{thm:existence}, for any nontrivial solution of BVP \eqref{eqn:el}-\eqref{eqn:bdry} we must have $C>2\E_B-1$. Since $\W_S^2< 2\E_B-1$ by the assumption, it follows \begin{equation*}
		g(0) = \W_S^2 < -1+2\E_B<C.
	\end{equation*}

To show that $\theta(x)$ is constant when the applied bias field is large enough, we recall that $\cos\theta(0)$ and $\cos\theta(1)$ are solutions of $g(t)=C$. If  $\E_B>1+\frac{\W_S^2}{\kappa}$, we have \begin{eqnarray*}
		g'({\color{black}t}) &=& 2\left[-\kappa\W_S^2\frac{{\color{black}t}}{(1-(1-\kappa)t^2)^2}- ({\color{black}t}-\E_B)\right]\\
&\geq&2\left[\E_B-1-\frac{\W_S^2}{\kappa}\right]>0 
	\end{eqnarray*} for {\color{black} all} $t\in [-1,1]$. 
Therefore $g(t)=C$ has at most one solution $t_0\in [-1,1]$, which correspond{\color{black} s} to a constant solution $\cos\theta(x)\equiv t_0$ of \eqref{eqn:el} according to \Cref{thm:existence}. Inserting $\theta\equiv 0$ and $\theta\equiv \pi$ into \eqref{eqn:enI}, we find \begin{equation*}
	I(0) = \frac{1}{2}\left(\frac{L}{\xi}\right)^2(1-2\tilde{E}_B) < \frac{1}{2}\left(\frac{L}{\xi}\right)^2(1+2\tilde{E}_B) = I(\pi).
\end{equation*} The proof is complete.
\end{proof}

{\color{black}
\begin{remark}
Lemma \ref{lm:bvbigWs} confirms that under {\color{black} sufficiently} strong surface anchoring, the polar director {\color{black} adheres to the polar anchoring condition $p_x(0)>0$, $p_x(1)<0$}. On the other hand, Lemma \ref{lm:theta1} says if the applied bias field is strong enough {\color{black} (i.e., $\E_B> \frac{1+\W_S^2}{2}$)}, the polar director tends to {\color{black} align} in the same direction as the applied field and eventually becomes constant when the field is very strong {\color{black} ($\E_B> 1+\frac{\W_S^2}{\kappa}$)}, suggesting that $p_{x}$ should approach the energy-minimizing fully aligned state $p_{x}\equiv 1$ as this latter threshold is reached. Moreover, {\color{black} this second bound suggests that for higher values of the bend-to-splay ratio $\kappa \approx 1$, a weaker applied bias field is needed in order to enforce the fully-aligned state $p_x\equiv 1$ in comparison to a smaller bend-to-splay ratio $\kappa\approx 0$}.
\end{remark}
}

{\color{black} 
\begin{remark}
	If $\tilde{E}_B > 0$ and $p_x\equiv 1$, \eqref{eqn:alg} implies that $C=2\E_B-1$. Recalling from \Cref{rmk:trivsol} that $p_x\equiv 1$ if and only if $C=2\E_B-1$, this provides a characterization of the fully-aligned state $p_x\equiv 1$ in terms of the integration constant from the first integral \eqref{eqn:fi}.
\end{remark}
}

Theorem \ref{thm:bdry} follows from Lemma \ref{lm:bvnec}, Lemma \ref{lm:bvbigWs} and Lemma \ref{lm:theta1}.

Our final result in this section shows that we can solve BVP \eqref{eqn:el}-\eqref{eqn:bdry} in the special case $\kappa=0$. For positive $\kappa$, we formalize the implicit solution obtained in \eqref{eqn:fiint} in the proof of \Cref{lm:bvnec}.

\begin{lemma}\label{lm:k}
{\color{black} The polar director} {\color{black} $\theta$ can be solved explicitly or implicitly depending on the} values of $\kappa$. {\color{black} Specifically,}
\begin{enumerate}
\item If $\kappa =0$, then $p_x(x)=\cos\theta(x)=A\sinh\frac{L}{\xi}x+B\cosh\frac{L}{{\color{black}\xi}}x+\E_B$, where
$$A={\color{black} -\W_S} ~~~\text{ and }~~~B={\color{black} \W_S \frac{\cosh\frac{L}{\xi}-1}{\sinh\frac{L}{\xi}}}.$$
\item If $\kappa>0$, then $\theta$ is implicitly defined by 
      $$
          \int_{\theta(0)}^{\theta(x)}\frac{{\color{black}\sqrt{\sin^2\theta+\kappa\cos^2\theta}}}{\sqrt{C+\cos^2\theta - 2\E_B\cos\theta}}d\theta=\frac{L}{\xi}x,
      $$
where $$C=\frac{\W_S^2\sin^2\theta(0)}{\sin^2\theta(0)+\kappa\cos^2\theta(0)}-\cos^2\theta(0)+2\E_B\cos\theta(0).$$
\end{enumerate}
\end{lemma}
\begin{proof}
\textbf{Case $\kappa=0$:}

If $\kappa=0$, \eqref{eqn:pxode} {\color{black} together with \eqref{eqn:bdry}} simplif{\color{black} y} to a linear equation for $p_x$:
$$
-p_x''+\left(\frac{L}{\xi}\right)^2(p_x-\E_B)=0{\color{black} }
$$
with boundary conditions 
$$
p_x'(0)=p_x'(1)=-\W_S {\color{black} \left(\frac{L}{\xi}\right)}. 
$$
Direct calculation {\color{black} shows} $p_x(x)=A\sinh\frac{L}{\xi}x+B\cosh\frac{L}{\xi}x+\E_B$, where $$A={\color{black} -\W_S} ~~~\text{ and }~~~B={\color{black} \W_S \frac{\cosh\frac{L}{\xi}-1}{\sinh\frac{L}{\xi}}}.$$

\textbf{Case $\kappa>0$:}

If $\kappa >0$, 
Multiply{\color{black} ing} $\theta'$ to the equation \eqref{eqn:el} and integrat{\color{black} ing} from $0$ to $x$, we get 
$$
\sqrt{\sin^2\theta+\kappa\cos^2\theta}\theta'(x)=\frac{L}{\xi}\sqrt{C+\cos^2\theta-2\E_B\cos\theta}.$$ Here $C=\frac{\W_S^2\sin^2\theta(0)}{\sin^2\theta(0)+\kappa\cos^2\theta(0)}-\cos^2\theta(0)+2\E_B\cos\theta(0)$ is determined by {\color{black} inserting the} boundary conditions \eqref{eqn:bdry} {\color{black} into \eqref{eqn:fi}}. The conclusion of the {\color{black} l}emma follows by integration. 
\end{proof}

\subsection{Effects of electric field, surface anchoring, and cell thickness}
In this section, we prove Theorem \ref{thm:eff}  which addresses  the effects of the bias electric field, surface anchoring, and cell thickness on the structure of the polar director. {\color{black} We again denote by $\theta\in C^\infty(0,1)$ a nonconstant solution of BVP \eqref{eqn:el}-\eqref{eqn:bdry}.} 

{\color{black} The following lemma concerns the effect of the applied bias field $\E_B$. For any nonconstant solution $\theta$ of BVP \eqref{eqn:el}-\eqref{eqn:bdry}, we define the average value $\langle p_x \rangle$ of the $x$-component $p_x(x)=\cos\theta(x)$ as} \begin{equation}\label{eqn:avval}
	\langle p_x \rangle = \int_0^1 p_x(x) dx.
\end{equation}

\begin{lemma}\label{lm:Eeff}
If $\E_B<1$ and $\W_S> \max\left\{\sqrt{\E_B^2+ \frac{1+\kappa}{2}\left(\frac{\pi\xi}{L}\right)^2}, \frac{\pi\xi}{L}\sqrt{\frac{1+\kappa}{2}\left(1+\frac{\kappa\E_B^2}{1-\E_B^2}\right)}\right\}$, then $p_x(0)$ increases as $\E_B$ increases. {\color{black} When $\tilde{E}_B=0$,} {\color{black} $\langle p_x \rangle =0$}.  For energy-minimizing configurations $p_x=\cos\theta$, {\color{black} {\color{black} $\langle p_x \rangle = 1$} when $\tilde{E}_B\to \infty$}.
\end{lemma}
\begin{proof}
We first prove the statement on {\color{black} the average value $\langle p_x \rangle$}. {\color{black} The Euler-Lagrange equation} \eqref{eqn:el} can be written  in terms of $p_x =\cos\theta$ and $p_y =\sin\theta$ as follows{\color{black} :}
\begin{equation}\label{eqn:pxode}
-\sin\theta \cdot p_x''(x)+\kappa\cos\theta \cdot p_y''(x)+\left(\frac{L}{\xi}\right)^2\sin\theta \cdot (p_x-\E_B)=0. 
\end{equation}
Divid{\color{black} ing} \eqref{eqn:pxode} by $\sin\theta${\color{black} ,} integrating by parts, {\color{black} and applying \eqref{eqn:bdry},} we get 
$$
\int_0^1 p_x(x) dx=\E_B-\kappa\left(\frac{\xi}{L}\right)^2\int_0^1\frac{\cos\theta(x)}{\sin^{2}\theta(x)}(\theta'(x))^2 dx. 
$$

When $\E_B=0$, by {\color{black} the} anti-symmetry of $\cos\theta(x)$ with respect to $x=\frac{1}{2}$ from Lemma \ref{lm:sym}, we get  
$$
\int_0^1\frac{\cos\theta(x)}{\sin^{2}\theta(x)}(\theta'(x))^2 dx=0.
$$
Therefore {\color{black} $\langle p_x \rangle =0$} {\color{black} when $\tilde{E}_B=0$.}

{\color{black} Since $\kappa\in (0,1)$ and $\W_S>0$ are fixed, the conclusion on $\langle p_x \rangle =1$ as $\E_B\to \infty$ follows from the assertion of the trivial solution $p_x\equiv 1$ in \Cref{lm:theta1} when $\E_B>1+\frac{\W_S^2}{\kappa}$.}

{\color{black} When $\tilde{E}_B>0$, we take} the first integral \eqref{eqn:id} of \eqref{eqn:el} {\color{black} to} get 
$$
 \left(\frac{\sin^2\theta(x)+\kappa \cos^2\theta(x)}{(\cos\theta(x)-\cos\theta(0))(\cos\theta(x)+\cos\theta(0)-2\E_B)+\frac{\W_S^2\sin^2\theta(0)}{\sin^2\theta(0)+\kappa\cos^2\theta(0)}}\right)^{\frac{1}{2}}d\theta=\left(\frac{L}{\xi}\right)dx.
$$
Integrat{\color{black} ing} from $x=0$ to $x=1$, {\color{black} we have}
\begin{equation}\label{eqn:tid}
\begin{split}
\frac{L}{\xi}&= \int_{p_x(1)}^{p_x(0)}\frac{\sqrt{1+(\kappa-1)y^2}}{\sqrt{1-y^2}\cdot\sqrt{(2\E_B-y-p_x(0))(p_x(0)-y)+\frac{\W_S^2\sin^2\theta(0)}{\sin^2\theta(0)+\kappa\cos^2\theta(0)}}}dy \\
	&= {\color{black} \int_{p_x(1)}^{p_x(0)} \sqrt{\frac{1+(\kappa-1)y^2}{1-y^2}}A(y) dy,}
\end{split}
\end{equation}
{\color{black} where $$
	A(y)= \frac{1}{\sqrt{(2\E_B-y-p_x(0))(p_x(0)-y)+ \frac{\W_S^2\sin^2\theta(0)}{\sin^2\theta(0)+\kappa\cos^2\theta(0)}}}.
$$}

To show {\color{black} that} $p_x(0)$ increases as $\E_B$ increases when $\E_B$ is in the stated range, set \begin{equation}\label{eqn:G}
\begin{split}
	G &= (2\E_B-y-p_x(0))(p_x(0)-y)+ {\color{black} \frac{\tilde{W}_S^2(1-p_x^2(0))}{1+(\kappa-1)p_x^2(0)}} \\ 
	&= {\color{black}  (2\E_B-y-p_x(1))(p_x(1)-y)+ {\color{black} \frac{\tilde{W}_S^2(1-p_x^2(1))}{1+(\kappa-1)p_x^2(1)}}}, 
\end{split}
\end{equation} where the equality follows from \eqref{eqn:three}. Differentiat{\color{black} ing} \eqref{eqn:tid} with respect to $\E_B$ on both sides, we obtain
\begin{eqnarray}\label{eqn:dE}
&&\frac{\sin^2\theta(0)+\kappa\cos^2\theta(0)}{\sin^2\theta(0)\W_S}\frac{dp_x(0)}{d\E_B}-\frac{\sin^2\theta(1)+\kappa\cos^2\theta(1)}{\sin^2\theta(1)\W_S}\frac{dp_x(1)}{d\E_B}\\ \nonumber 
&=&\int_{p_x(1)}^{p_x(0)}\sqrt{1+\frac{\kappa y^2}{1-y^2}}G^{-\frac{3}{2}}\cdot \frac{1}{2}\frac{dG}{d\E_B} dy,
\end{eqnarray}
where 
\begin{equation}\label{eqn:dG}
\frac{dG}{d\E_B}=2p_x(0)-2y+\left[2\E_B-2p_x(0)-\frac{2\kappa \W_S^2p_x(0)}{({\color{black} 1+(\kappa-1)p_x^2(0)})^2}\right]\frac{dp_x(0)}{d\E_B}.
\end{equation}
Differentiat{\color{black} ing} \eqref{eqn:three} with respect to $\E_B$ {\color{black} and dividing by $2$}, we obtain
\begin{eqnarray}\label{eqn:dbd}
&&\left[\E_B-p_x(0)-\frac{\kappa \W^2_Sp_x(0)}{({\color{black} 1+(\kappa-1)p_x^2(0)})^2}\right]\frac{dp_x(0)}{d\E_B}+p_x(0)-p_x(1)\\ \nonumber 
&=&\left[\E_B-p_x(1)-\frac{\kappa \W^2_Sp_x(1)}{({\color{black} 1+(\kappa-1)p_x^2(1)})^2}\right]\frac{dp_x(1)}{d\E_B}.
\end{eqnarray}
Substitut{\color{black} ing}  \eqref{eqn:dG}-\eqref{eqn:dbd} into \eqref{eqn:dE}, we get 
\begin{eqnarray}\label{eqn:dpx0}
&&\frac{dp_x(0)}{d\E_B}\cdot H\\ \nonumber 
&=& {\color{black} \W_S}\int_{p_x(1)}^{p_x(0)}\sqrt{1+\frac{\kappa y^2}{1-y^2}}G^{-\frac{3}{2}}(p_x(0)-y)dy+\frac{{\color{black} 1+(\kappa-1)p_x^2(1)}}{{\color{black} 1-p_x^2(1)}}\cdot\frac{p_x(0)-p_x(1)}{{\color{black}h}(p_x(1))}
\end{eqnarray}
where 
\begin{equation}\label{eqn:f}
{\color{black}h}(x)=\E_B-x-\frac{\kappa x\W^2_S}{(1+(\kappa-1)x^2)^2}
\end{equation}
and
\begin{eqnarray}\label{eqn:H}
H&=&\frac{{\color{black} 1+(\kappa-1)p_x^2(0)}}{{\color{black} 1-p_x^2(0)}}-\frac{{\color{black} 1+(\kappa-1)p_x^2(1)}}{{\color{black} 1-p_x^2(1)}}\cdot\frac{{\color{black}h}(p_x(0))}{{\color{black}h}(p_x(1))}\\ \nonumber 
&&-\int_{p_x(1)}^{p_x(0)}\sqrt{1+\frac{\kappa y^2}{1-y^2}}G^{-\frac{3}{2}}dy\cdot {\color{black}h}(p_x(0)).
\end{eqnarray}
{\color{black} Since $\E_B<1$ and $\W_S> \max\left\{\sqrt{\E_B^2 + \frac{1+\kappa}{2}\left(\frac{\pi\xi}{L}\right)^2}, \frac{\pi\xi}{L}\sqrt{\frac{1+\kappa}{2}\left(1+ \frac{\kappa\E_B^2}{1-\E_B^2}\right)}\right\}$, \Cref{lm:bvbigWs} implies that $p_x(0)>\E_B$ and $p_x(1)<0$, so that ${\color{black}h}(p_x(0))<0$ and ${\color{black}h}(p_x(1))>0$. It then follows that} {\color{black} $H>0$.} {\color{black} Since} the right hand side of \eqref{eqn:dpx0} is positive, the conclusion on $\frac{dp_x(0)}{d\E_B}$ follows.
\end{proof}

{\color{black} The following two lemmas describe the effects of the surface anchoring and cell thickness.}

\begin{lemma}\label{lm:Wsasymp}
Fix $\E_B=0$ and let  $\W_S$ increase from $0$ to $\infty${\color{black} . T}hen $p_x(0)$ increases from $0$ to $1$ while  $p_x(1)$ decreases from $0$ to $-1$.
\end{lemma}
\begin{proof}
{\color{black}
We first insert the boundary condition \eqref{eqn:bdry} into \eqref{eqn:fi} to obtain
\begin{eqnarray}\label{eqn:two}
&&\frac{\W^2_S}{(1-\kappa)+\frac{\kappa} {\sin^2\theta(0)}}-(\cos^2\theta(0)-2\E_B\cos\theta(0))\\ \nonumber 
&=&\frac{\W^2_S}{(1-\kappa)+\frac{\kappa} {\sin^2\theta(1)}}-(\cos^2\theta(1)-2\E_B\cos\theta(1)).
\end{eqnarray}
Thus} {\color{black} for any $\E_B\ge 0$} we have 
\begin{eqnarray}\label{eqn:cosbd}
&&\frac{\kappa\W^2_S}{w(\cos\theta({\color{black} 0}))\cdot w(\cos\theta(1))}\cdot (\cos\theta(1)-\cos\theta(0))(\cos\theta(0)+\cos\theta(1))\\ \nonumber 
&&=(\cos\theta(1)-\cos\theta(0))(2\E_B-\cos\theta(1)-\cos\theta(0))
\end{eqnarray}
where $w({\color{black}t})=(1+(\kappa-1){\color{black}t}^2)$. If $\cos\theta(1)-\cos\theta(0)\neq 0$, \eqref{eqn:cosbd} yields $0\leq \cos\theta(1)+\cos\theta(0)\leq 2\E_B.$ If $\W_S=0$, {\color{black} since $\theta^\prime>0$ we must have 
$$
	2\E_B - \cos\theta(1)-\cos\theta(0)=0
$$ by \eqref{eqn:cosbd}, that is,
\begin{equation}\label{eqn:cos01}
	\cos\theta(0)-\E_B = -\cos\theta(1)+ \tilde{E}_B.	
\end{equation}}

When $\E_B=0$ and $\W_S=0$, {\color{black}\eqref{eqn:tid} and \eqref{eqn:cos01} imply that} $\theta(x)\equiv \frac{\pi}{2}$. When $\W_S\rightarrow \infty$, since the {\color{black}left} hand side of \eqref{eqn:tid} is constant, we must have $\sin\theta(0)\rightarrow 0$ and $\sin\theta(1)\rightarrow {\color{black}0}$, i.e. $p_x(0)\rightarrow 1$ and $p_x(1)\rightarrow -1$ by \eqref{eqn:cos01}. 

To show how $p_x(0)$ changes when $\W_S$ increases, we differentiate \eqref{eqn:tid} with respect to $\W_S$ on both sides and get 
\begin{eqnarray}\label{eqn:dW}
&&\frac{\sin^2\theta(0)+\kappa\cos^2\theta(0)}{\W_S\sin^2\theta(0)}\frac{dp_x(0)}{d\W_s}-\frac{\sin^2\theta(1)+\kappa\cos^2\theta(1)}{\W_S\sin^2\theta(1)}\frac{dp_x(1)}{d\W_s}\\ \nonumber
&=&\int_{p_x(1)}^{p_x(0)}\sqrt{1+\frac{\kappa y^2}{1-y^2}}G^{-\frac{3}{2}}\cdot \frac{1}{2}\frac{dG}{d\W_S} dy{\color{black} ,}
\end{eqnarray}
{\color{black} where $G$ is defined by \eqref{eqn:G},}
\begin{eqnarray}\label{eqn:dGW}
\frac{dG}{d\W_S}&=&2\left[\frac{\W_S}{1-\kappa+\frac{\kappa}{\sin^2\theta(0)}}+{\color{black}h}(\cos\theta(0))\frac{dp_x(0)}{d\W_S}\right]\\ \nonumber
&=&2\left[\frac{\W_S}{1-\kappa+\frac{\kappa}{\sin^2\theta(1)}}+{\color{black}h}(\cos\theta(1))\frac{dp_x(1)}{d\W_S}\right],
\end{eqnarray}
{\color{black} and} ${\color{black}h}(x)$ is defined by \eqref{eqn:f}.
Combining \eqref{eqn:dW} and \eqref{eqn:dGW}, we get 
\begin{eqnarray}\label{eqn:p0W}
\frac{dp_x(0)}{d\W_S}\cdot H&=&\frac{\W^2_S}{1-\kappa+\frac{\kappa}{\sin^2\theta(0)}}\int_{p_x(1)}^{p_x(0)}\sqrt{1+\frac{\kappa y^2}{1-y^2}}G^{-\frac{3}{2}}dy\\ \nonumber 
&&+\frac{\sin^2\theta(1)+\kappa\cos^2\theta(1)}{{\color{black}h}(\cos\theta(1))\sin^2\theta(1)}\left(\frac{\W^2_S}{1-\kappa+\frac{\kappa}{\sin^2\theta(0)}}-\frac{\W^2_S}{1-\kappa+\frac{\kappa}{\sin^2\theta(1)}}\right),
\end{eqnarray}
where $H$ is defined by \eqref{eqn:H}.
Since $\tilde{E}_B=0$, \Cref{lm:mono} and \Cref{lm:sym} imply that $\cos\theta(1) < 0 < \cos\theta(0)$ for a nonconstant solution $\theta$. Consequently, ${\color{black}h}(\cos\theta(0))<0,$ ${\color{black}h}(\cos\theta(1))>0$, and $\eqref{eqn:H}$ gives $H>0$. It then follows that the right hand side of  \eqref{eqn:p0W} is positive, which shows that $\frac{dp_x(0)}{d\W_S}>0$. In addition, since $\sin^2\theta(0)=\sin^2\theta(1)$, \eqref{eqn:dGW} yields $\frac{dp_x(1)}{d\W_S}<0$.


\end{proof}

\begin{lemma}\label{lm:lxipx0}
Fix $\E_B<1$ and $\W_S> \max\left\{\sqrt{\E_B^2+ \frac{1+\kappa}{2}\left(\frac{\pi\xi}{L}\right)^2}, \frac{\pi\xi}{L}\sqrt{\frac{1+\kappa}{2}\left(1+\frac{\kappa\E_B^2}{1-\E_B^2}\right)}\right\}$. If $\frac{L}{\xi}$ increases, then $p_x(0)$ increases {\color{black} and}  $p_x(1)$ decreases.
\end{lemma}
\begin{proof}
Differentiat{\color{black} ing} \eqref{eqn:tid} with respect to $\frac{L}{\xi}$ on both sides, we get 
\begin{eqnarray}\label{eqn:dL}
&&\frac{\sin^2\theta(0)+\kappa\cos^2\theta(0)}{\W_S\sin^2\theta(0)}\frac{dp_x(0)}{d{\frac{L}{\xi}}}-\frac{\sin^2\theta(1)+\kappa\cos^2\theta(1)}{\W_S\sin^2\theta(1)}\frac{dp_x(1)}{d{\frac{L}{\xi}}}\\ \nonumber 
&=&\int_{p_x(1)}^{p_x(0)}\sqrt{1+\frac{\kappa y^2}{1-y^2}}G^{-\frac{3}{2}}\cdot \frac{1}{2}\frac{dG}{d{\frac{L}{\xi}}} dy+1,
\end{eqnarray}
where 
\begin{eqnarray}\label{eqn:dGL}
\frac{dG}{d\frac{L}{\xi}}&=&\left[2\E_B-2p_x(0)-\frac{2\kappa \W_S^2p_x(0)}{(\sin^2\theta(0)+\kappa\cos^2\theta(0))^2}\right]\frac{dp_x(0)}{d\frac{L}{\xi}}\\ \nonumber 
                                    &=&\left[2\E_B-2p_x(1)-\frac{2\kappa\W_S^2 p_x(1)}{(\sin^2\theta(1)+\kappa\cos^2\theta(1))^2}\right]\frac{dp_x(1)}{d\frac{L}{\xi}}
\end{eqnarray}
Combining \eqref{eqn:dL} and \eqref{eqn:dGL}, we get {\color{black} \begin{equation}\label{eqn:I}
	I\cdot \frac{dp_x(0)}{d\frac{L}{\xi}} = 1,
\end{equation}
where
\begin{equation*}
\begin{split}
	I &= \frac{\sin^2\theta(0)+\kappa\cos^2\theta(0)}{\W_S\sin^2\theta(0)}- \frac{\sin^2\theta(1)+\kappa\cos^2\theta(1)}{\W_S\sin^2\theta(1)} \cdot \frac{{\color{black}h}(p_x(0))}{{\color{black}h}(p_x(1))} 
		\\ &\qquad -  \int_{p_x(1)}^{p_x(0)} \sqrt{1+ \frac{\kappa y^2}{1-y^2}} G^{-\frac{3}{2}} dy \cdot {\color{black}h}(p_x(0))
\end{split}
\end{equation*}
and $h$ is defined in \eqref{eqn:f}. {\color{black} Since $\E_B<1$ and $\W_S> \max\left\{\sqrt{\E_B^2+ \frac{1+\kappa}{2}\left(\frac{\pi\xi}{L}\right)^2}, \frac{\pi\xi}{L}\sqrt{\frac{1+\kappa}{2}\left(1+\frac{\kappa\E_B^2}{1-\E_B^2}\right)}\right\}$, \Cref{lm:bvbigWs} implies that $p_x(0)>\E_B$ and $p_x(1)<0$, so again ${\color{black}h}(p_x(0))<0$ and ${\color{black}h}(p_x(1))>0$.} Hence, $I>0$. In view of \eqref{eqn:I}, $\frac{dp_x(0)}{d\frac{L}{\xi}}>0$. It then follows from equation \eqref{eqn:dGL} that $\frac{dp_x(1)}{d\frac{L}{\xi}}<0$. }

\end{proof}

\Cref{lm:Eeff,lm:Wsasymp,lm:lxipx0} together prove \Cref{thm:eff}.

\subsection{Local stability of $\theta$}
{\color{black} In this section, we consider the local stability of nonconstant solutions to BVP \eqref{eqn:el}-\eqref{eqn:bdry}. In particular, we prove \Cref{thm:stability} in this section.}

\begin{lemma}\label{lm:st}
If $\sin^2\theta\geq \frac{\kappa}{1{\color{black}+}\kappa}$ and $\E_B =0$, then $H_{\theta}(\varphi,\varphi)\geq 0$ for any $\varphi$ satisfying $\varphi(0)=\varphi(1)=0$. In particular, if $\kappa=0$ and $\E_B=0$, then $H_{\theta}(\varphi,\varphi)\geq 0$ for any $\varphi$ vanishing on the boundary. 
\end{lemma}
\begin{proof}
Direct calculation shows that 
\begin{eqnarray}\label{eqn:Hess}
2H_{\theta}(\varphi,\varphi)&=&2\frac{d^2}{d^2\varepsilon}I(\theta+ \varepsilon\varphi)\bigg|_{\varepsilon=0}\\ \nonumber
                                        &=&\int_0^1\left((1-\kappa)(\theta')^2\varphi^2\cos2\theta +(\sin^2\theta+\kappa\cos^2\theta)(\varphi')^2-2(\kappa-1)\theta'\varphi'\varphi\sin2\theta\right)dx\\ \nonumber 
                                        &&+\left(\frac{L}{\xi}\right)^2\int_0^1(-\cos2\theta\varphi^2+\E_B\cos\theta\varphi^2)dx+\W_S\frac{L}{\xi}\left(\cos\theta(0)\varphi^2(0)-\cos\theta(1)\varphi^2(1)\right).
\end{eqnarray}
By integration by parts, we have 
\begin{eqnarray}
&& -\int_0^1 2\theta'\varphi'\varphi\sin2\theta dx\\ \nonumber 
&=&\int_0^1\varphi^2\theta''\sin2\theta dx+2\int_0^1\varphi^2\theta'^2\cos2\theta{\color{black}dx-}\varphi^2\theta'\sin2\theta\bigg|_0^1\\ \nonumber
&=&-\int_0^1\frac{\varphi^2}{\sin^2\theta+\kappa\cos^2\theta}\left[\frac{1}{2}\sin^22\theta(1-\kappa)\theta'^2+\frac{1}{2}\sin^22\theta\left(\frac{L}{\xi}\right)^2-\left(\frac{L}{\xi}\right)^2\E_B\sin2\theta\sin\theta\right]\\ \nonumber 
&&+2\int_0^1\varphi^2\theta'^2\cos2\theta dx-\varphi^2\theta'\sin2\theta\bigg|_0^1 . 
\end{eqnarray}

Plugging back into \eqref{eqn:Hess} {\color{black} and inserting the boundary values for $\theta^\prime$ obtained from \eqref{eqn:bdry}}, we have
\begin{eqnarray}
2H_{\theta}(\varphi,\varphi)&=&\int_0^1\left(-(1-\kappa)(\theta')^2\varphi^2\cos2\theta +(\sin^2\theta+\kappa\cos^2\theta)\varphi'^2\right) dx\\ \nonumber
                                        &&+(1-\kappa)^2\int_0^1\frac{\varphi^2\theta'^2}{\sin^2\theta+\kappa\cos^2\theta}\frac{1}{2}\sin^22\theta dx-\W_S\left(\frac{L}{\xi}\right)\cos\theta\varphi^2\bigg|_0^1+(1-\kappa)\varphi^2\theta'\sin2\theta\bigg|_0^1 \\ \nonumber 
                                        &&+\left(\frac{L}{\xi}\right)^2\int_0^1\varphi^2\left(\frac{\frac{1}{2}{\color{black} (\kappa-1)}}{\sin^2\theta+\kappa\cos^2\theta}(-\sin^22\theta+{\color{black}2}\E_B\sin2\theta\sin\theta)-\cos2\theta+\E_B\cos\theta\right)dx\\ \nonumber 
               &=&\int_0^1(\sin^2\theta+\kappa\cos^2\theta)\varphi'^2dx+(1-\kappa)\int_0^1\theta'^2\varphi^2\left[-\cos2\theta+\frac{\frac{1}{2}(1-\kappa)\sin^22\theta}{\sin^2\theta+\kappa\cos^2\theta}\right]dx\\ \nonumber 
&&+\left(\frac{L}{\xi}\right)^2\int_0^1\varphi^2\left[-\cos2\theta+\frac{\frac{1}{2}(1-\kappa)\sin^22\theta}{\sin^2\theta+\kappa\cos^2\theta}\right]dx\\ \nonumber 
&&+\left(\frac{L}{\xi}\right)^2\E_B\int_0^1\varphi^2\left[\cos\theta-\frac{(1-\kappa)\sin2\theta\sin\theta}{\sin^2\theta+\kappa\cos^2\theta}\right]dx\\ \nonumber 
&&{\color{black}+\W_S\left(\frac{L}{\xi}\right){\color{black}\cos\theta}\varphi^2\left[\frac{\sin^2\theta(1-\kappa)-\kappa}{\sin^2\theta+\kappa\cos^2\theta}\right]\bigg|_0^1}\\ \nonumber
&=&\int_0^1(\sin^2\theta+\kappa\cos^2\theta)\varphi'^2dx+(1-\kappa)\int_0^1\theta'^2\varphi^2\frac{\sin^2\theta(1+\kappa)-\kappa}{\sin^2\theta+\kappa\cos^2\theta}dx \\ \nonumber
&&+\left(\frac{L}{\xi}\right)^2\int_0^1\varphi^2\frac{\sin^2\theta(1+\kappa)-\kappa}{\sin^2\theta+\kappa\cos^2\theta}dx+\left(\frac{L}{\xi}\right)^2\E_B\int_0^1\varphi^2\frac{\cos\theta(\kappa-(1-\kappa)\sin^2\theta)}{\sin^2\theta+\kappa\cos^2\theta}dx\\ \nonumber 
&&+\W_S\left(\frac{L}{\xi}\right){\color{black}\cos\theta}\varphi^2\left[\frac{\sin^2\theta(1-\kappa)-\kappa}{\sin^2\theta+\kappa\cos^2\theta}\right]\bigg|_0^1.
\end{eqnarray}
{\color{black} The} conclusion of the lemma follows. 
\end{proof}

\begin{lemma}\label{lm:sinxbd}
If $\E_B=0$ and $\W_S^2<\frac{2}{1+\kappa}$, then $\sin^2\theta(x) \ge \frac{\kappa}{1+\kappa}$ for all $x\in [0,1]$.
\end{lemma}
\begin{proof}
We shall show $\sin^2\theta(0)\ge \frac{\kappa}{1+\kappa}$ under the assumption $\W^2_S<\frac{2}{1+\kappa}$ {\color{black}and} $\E_B=0$. Recall {\color{black} from \eqref{eqn:id} and \eqref{eqn:bdry}} that 
$$
\begin{array}{l}
(\sin^2\theta+\kappa\cos^2\theta)\theta'^2-\left(\frac{L}{\xi}\right)^2(\cos^2\theta-2\cos\theta \E_B)\\
=\left(\frac{L}{\xi}\right)^2\left[\frac{\W^2_S\sin^2\theta(0)}{\sin^2\theta(0)+\kappa\cos^2\theta(0)}-(\cos^2\theta(0)-2\cos\theta(0) \E_B)\right].
\end{array}
$$
{\color{black}When $\E_B=0$, Lemma \ref{lm:sym} says {\color{black} that} $\theta (\frac{1}{2})=\frac{\pi}{2}${\color{black}.} {\color{black} P}lugging in $x=\frac{1}{2}$} to the left hand side, we get 
$$
\frac{\W^2_S\sin^2\theta(0)}{\sin^2\theta(0)+\kappa\cos^2\theta(0)}-\cos^2\theta(0)\geq 0.
$$
Set{\color{black}ting} $y=\cos^2\theta(0)$ in the inequality above, we get
$$
(1-\kappa)y^2-(\W^2_S+1)y+\W^2_S\geq 0.
$$
Solving this, we {\color{black} obtain}

$$
y\geq \frac{\W^2_S+1+\sqrt{({\color{black} \W_S ^2}-1)^2+4\kappa\W^2_S}}{2(1-\kappa)} ~~~\text{  or   }~~~ y\leq\frac{\W^2_S+1-\sqrt{({\color{black} \W_S ^2}-1)^2+4\kappa\W^2_S}}{2(1-\kappa)}.
$$

Therefore 
$$
\sin^2\theta(0)=1-y\geq 1-\frac{\W^2_S+1-\sqrt{({\color{black} \W_S ^2}-1)^2+4\kappa\W^2_S}}{2(1-\kappa)}.
$$
Solv{\color{black} ing}
$$
1-\frac{\W^2_S+1-\sqrt{({\color{black} \W_S ^2}-1)^2+4\kappa\W^2_S}}{2(1-\kappa)}\geq \frac{\kappa}{1+\kappa},
$$
we get 
$$
\sqrt{({\color{black} \W_S ^2}-1)^2+4\kappa\W^2_S}\geq \W^2_S+1-2\frac{1-{\color{black}\kappa}}{1+\kappa}.
$$
which is equivalent to 
$$
4\kappa\W^2_S\geq 8\frac{\kappa}{1+\kappa}(\W^2_S-1)+\left(\frac{4\kappa}{1+\kappa}\right)^2.
$$
{\color{black} Solving this for $\W_S^2$, we find that}
$
\W^2_S\leq \frac{2}{1+\kappa}. 
$
Therefore $\sin^2\theta(x)\geq\sin^2\theta(0)$ for all $x\in [0,1]$ by monotonicity of $\theta$.
\end{proof}
Theorem \ref{thm:stability} follows directly from Lemma \ref{lm:st} {\color{black} and \Cref{lm:sinxbd}.}

\section{Numerical simulations}\label{sec:num}
In this section, we consider numerical solutions of boundary value problem \eqref{eqn:el}-\eqref{eqn:bdry} at different values of the polar surface anchoring $\tilde{W}_S$, external bias electric field $\tilde{E}_B$, bend-to-splay ratio $\kappa$, and cell thickness $L/\xi$. Our results are shown in \Cref{fig:Eb-Lxi,fig:Ws-Lxi,fig:Eb-Ws,fig:Ws-Eb,fig:Eb-k,fig:avval,fig:Eb-decomp,fig:Eb-norm-decomp}.

\begin{figure}[!htb]
\centering
\begin{subfigure}{0.49\textwidth}
	\includegraphics[width=\textwidth]{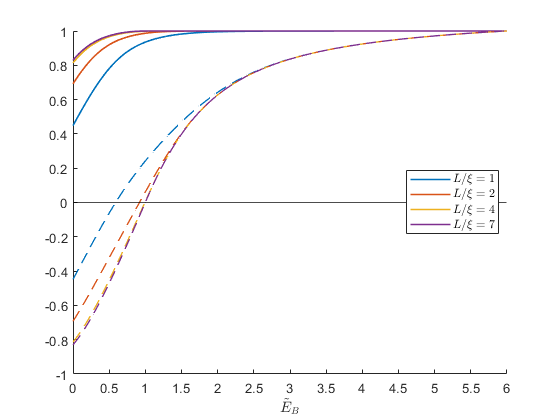}
	\caption{}
	\label{subfig:bdry_Eb-Lxi}
\end{subfigure}
\hfill
\begin{subfigure}{0.49\textwidth}
	\includegraphics[width=\textwidth]{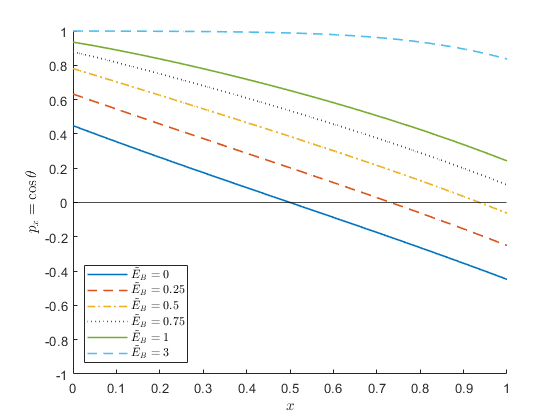}
	\caption{$L/\xi=1$}
	\label{subfig:Eb-Lxi1}
\end{subfigure}

\begin{subfigure}{0.49\textwidth}
	\includegraphics[width=\textwidth]{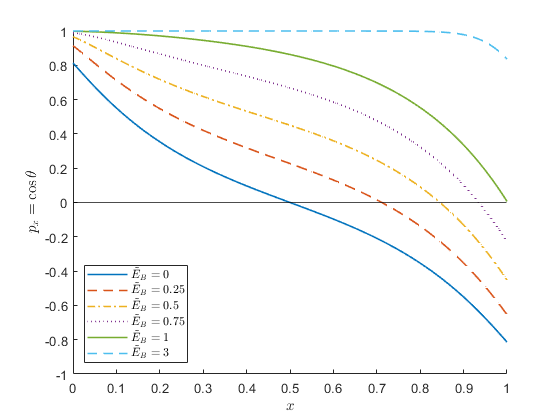}
	\caption{$L/\xi=4$}
	\label{subfig:Eb-Lxi2}
\end{subfigure}
\hfill
\begin{subfigure}{0.49\textwidth}
	\includegraphics[width=\textwidth]{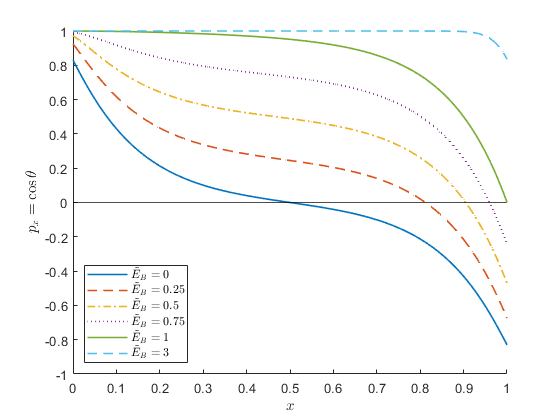}
	\caption{$L/\xi=7$}
	\label{subfig:Eb-Lxi3}
\end{subfigure}

\caption{(a) The boundary values $p_x(0)=\cos\theta(0)$ (solid) and $p_x(1)=\cos\theta(1)$ (dashed) at different cell thicknesses $L/\xi$ as a function of the bias field $\E_B$; (b), (c), (d): The spatial profiles of $p_x=\cos\theta$ at different cell thicknesses and different biases. Parameter values: $\kappa=0.2$ and $\tilde{W}_S=1$.}
\label{fig:Eb-Lxi}
\end{figure}

\begin{figure}[!tb]
\centering
\begin{subfigure}{0.49\textwidth}
	\includegraphics[width=\textwidth]{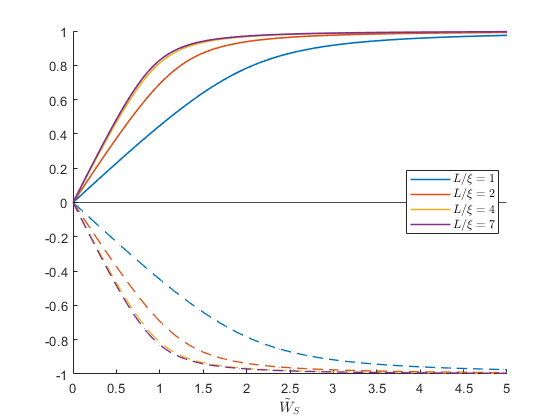}
	\caption{}
	\label{subfig:bdry_Ws-Lxi}
\end{subfigure}
\hfill
\begin{subfigure}{0.49\textwidth}
	\includegraphics[width=\textwidth]{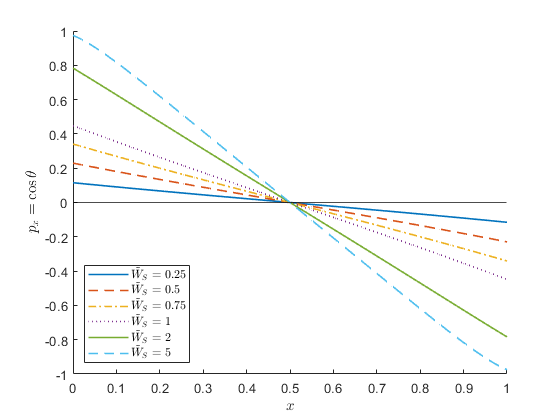}
	\caption{$L/\xi=1$}
	\label{subfig:Ws-Lxi1}
\end{subfigure}

\begin{subfigure}{0.49\textwidth}
	\includegraphics[width=\textwidth]{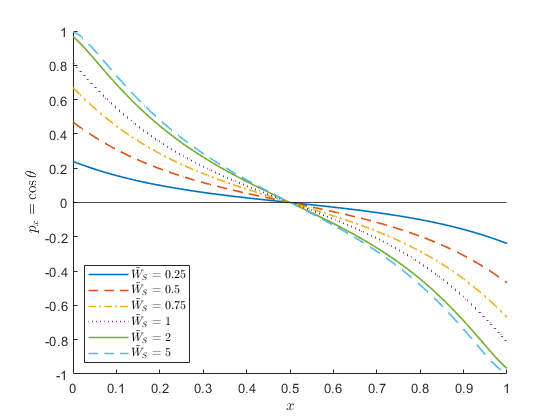}
	\caption{$L/\xi=4$}
	\label{subfig:Ws-Lxi2}
\end{subfigure}
\hfill
\begin{subfigure}{0.49\textwidth}
	\includegraphics[width=\textwidth]{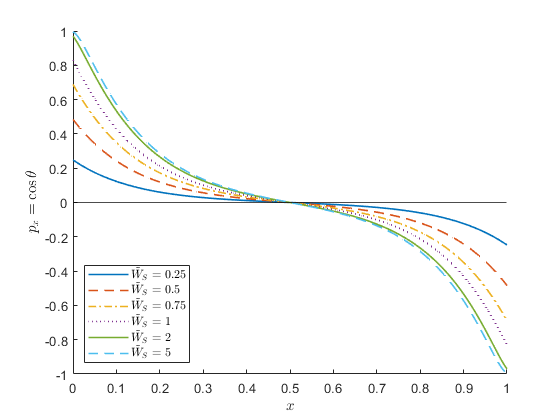}
	\caption{$L/\xi=7$}
	\label{subfig:Ws-Lxi3}
\end{subfigure}

\caption{(a) The boundary values $p_x(0)= \cos\theta(0)$ (solid) and $p_x(1)=\cos\theta(1)$ (dashed) at different cell thicknesses $L/\xi$ as a function of the anchoring strength $\tilde{W}_S$; (b), (c), (d): The spatial profiles of $p_x=\cos\theta$ at different anchoring strengths and different cell thicknesses. Parameter values: $\kappa=0.2$ and $\tilde{E}_B=0$.}
\label{fig:Ws-Lxi}
\end{figure}

To obtain numerical solutions to BVP \eqref{eqn:el}-\eqref{eqn:bdry}, we consider the $L^2[0,1]$ gradient flow ${\color{black}\frac{\partial\theta}{\partial t}} = -\frac{\delta I}{\delta \theta}$ for the nondimensional energy functional $I$ as defined in \eqref{eqn:enI}. The resulting flow is \begin{equation}\label{eqn:gd}
	\begin{cases}
		{\color{black}\dfrac{\partial\theta}{\partial t}} &= \big(\sin^2\theta +\kappa\cos^2\theta\big)\theta^{\prime\prime}+ \dfrac{1}{2}(1-\kappa)\sin2\theta \big(\theta^\prime\big)^2 +  \dfrac{1}{2}\Big(\dfrac{L}{\xi}\Big)^2\big(\sin 2\theta-2\tilde{E}_B\sin\theta\big), \\ 
		
		{\color{black} 0} &= \left[-\W_S \left(\dfrac{L}{\xi}\right)\sin\theta(x) + \left(\sin^2\theta(x)+ \kappa\cos^2\theta(x)\right)\theta^\prime(x) \right]\biggr\rvert_{x=0,1}.
	\end{cases}
\end{equation} We discretize the system \eqref{eqn:gd} using a first-order implicit scheme for the time variable (i.e., backward Euler or gradient descent), and a standard second-order finite-difference scheme for the space variable. {\color{black} The nonlinear Robin boundary conditions were implemented through a ``ghost-node" formulation consistent with the second--order explicit discretization to preserve the expected second--order convergence in space.} We solve the resulting nonlinear system of equations via Newton's method, and terminate the iteration process when {\color{black} the relative infinity norm of successive approximations fall below a prescribed tolerance $10^{-8}$.}  {\color{black} Computations were performed on successively refined meshes, with the resulting profiles exhibiting second-order numerical convergence in both $L^2[0,1]$ and $L^\infty[0,1]$ norm.} {\color{black} Our simulations are also consistent with previous calculations depicted in \cite{GCV14,Les16}.} All results are implemented in MATLAB Version: 24.2.0.2833386 (R2024b).

\begin{figure}[!tb]
\centering

\begin{subfigure}{0.49\textwidth}
	\includegraphics[width=\textwidth]{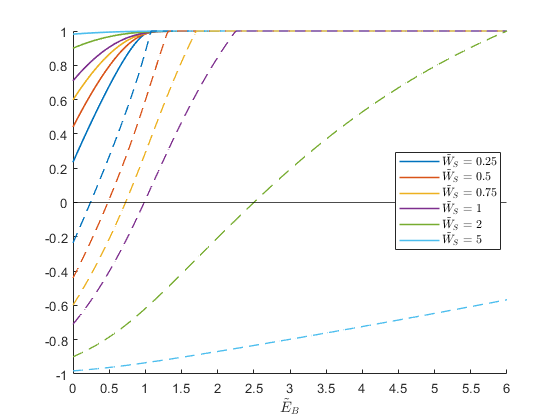}
	\caption{}
	\label{subfig:bdry_Eb-Ws}
\end{subfigure}	
\hfill
\begin{subfigure}{0.49\textwidth}
	\includegraphics[width=\textwidth]{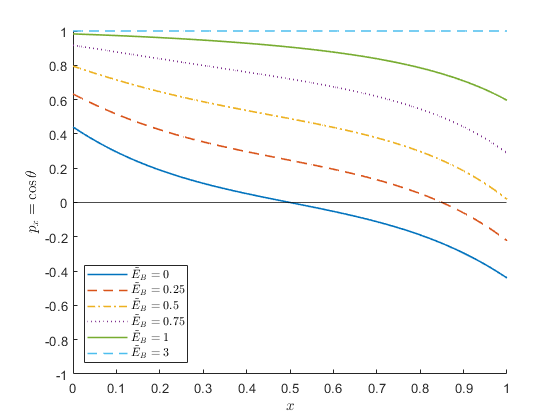}
	\caption{$\tilde{W}_S=0.5$}
	\label{subfig:Eb-Ws1}
\end{subfigure}

\begin{subfigure}{0.49\textwidth}
	\includegraphics[width=\textwidth]{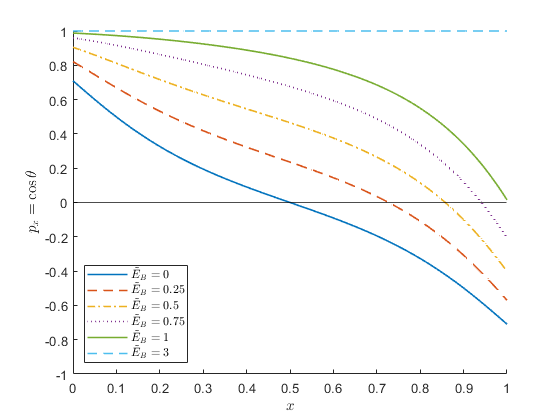}
	\caption{$\tilde{W}_S=1$}
	\label{subfig:Eb-Ws2}
\end{subfigure}
\hfill
\begin{subfigure}{0.49\textwidth}
	\includegraphics[width=\textwidth]{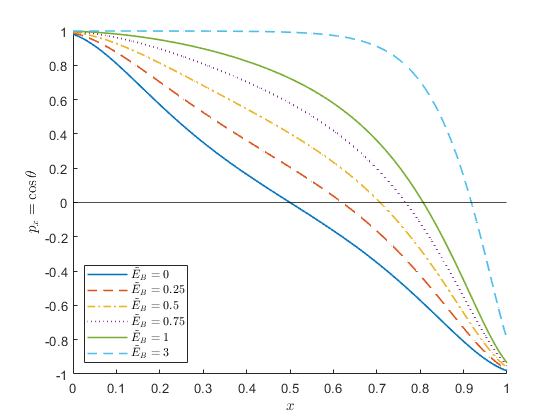}
	\caption{$\tilde{W}_S=5$}
	\label{subfig:Eb-Ws3}
\end{subfigure}

\caption{(a) The boundary values $p_x(0)= \cos\theta(0)$ (solid) and $p_x(1)=\cos\theta(1)$ (dashed) at different anchoring stengths $\W_S$ as a function of the bias field $\E_B$; (b), (c), (d): The spatial profiles of $p_x= \cos\theta$ at different biases $\tilde{E}_B$ at (b) weak [$\tilde{W}_S=0.5$], (c) intermediate [$\tilde{W}_S=1$], and (d) strong [$\tilde{W}_S=5$] anchorings. Parameter values: $\kappa=0.8$ and $L/\xi=4$.}
\label{fig:Eb-Ws}
\end{figure}

\Cref{fig:Ws-Lxi,fig:Eb-Lxi,fig:Eb-Ws,fig:Ws-Eb,fig:Eb-k} depict the boundary values $p_x(0)$ and $p_x(1)$ as well as representative spatial profiles $p_x=\cos\theta$ at different values of the bend-to-splay ratio $\kappa$, cell thickness $L/\xi$, applied bias $\E_B$, and anchoring strength $\W_S$. {\color{black} In \Cref{fig:Ws-Lxi,fig:Eb-Lxi}, we choose the representatives $L/\xi=1$, $L/\xi=4$, and $L/\xi=7$ to illustrate the individual effects of the anchoring strength and applied bias at different length scales. In view of \eqref{eqn:xi}, $L/\xi$ increases as the magnitude of the ferroelectric polarization becomes increasingly large in comparison to the elastic defects; therefore, we may interpret the transition from $L/\xi=1$ to $L/\xi=7$ as a transition from an elastically-dominant regime to a more electrostatically-dominant regime. In \Cref{fig:Eb-Ws,fig:Ws-Eb,fig:Eb-k}, we choose an intermediate value $L/\xi=4$ for the cell thickness to demonstrate the competition between the surface anchoring and applied bias. Since Gornik, \v Cepi\v c and Vaupoti\v{c} \cite{GCV14} define weak anchoring as $p_x(0)<0.5$ and strong anchoring as $p_x(0)>0.9$ when $\E_B=0$, we choose the representatives $\W_S=0.5$, $\W_S=1$, and $\W_S=5$ in \Cref{fig:Eb-Ws} to include profiles from each case. Similarly, we choose the values $\E_B=0$, $\E_B=0.5$, and $\E_B=1$ in \Cref{fig:Ws-Eb} to depict spatial profiles in the presence of zero bias, moderate bias, and strong bias, respectively.} {\color{black} \Cref{fig:Eb-k} is included to illustrate the effect of the bend-to-splay ratio $\kappa$ in \Cref{lm:theta1}.}

\begin{figure}[!tb]
\centering

\begin{subfigure}{0.49\textwidth}
	\includegraphics[width=\textwidth]{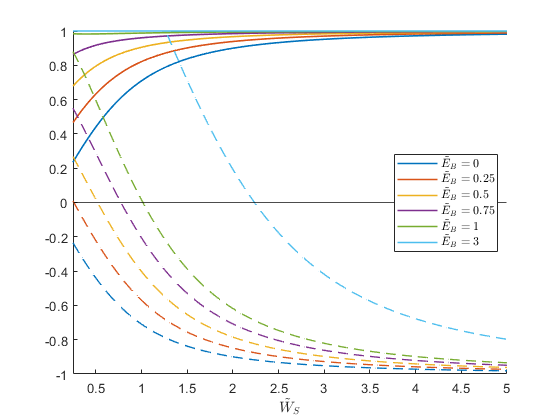}
	\caption{}
	\label{subfig:bdry_Ws-Eb}
\end{subfigure}
\hfill
\begin{subfigure}{0.49\textwidth}
	\includegraphics[width=\textwidth]{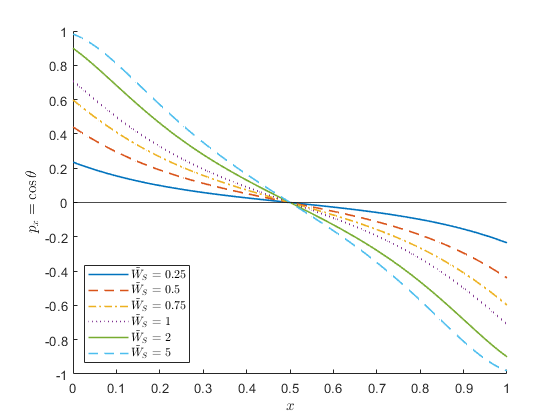}
	\caption{$\tilde{E}_B=0$}
	\label{subfig:Ws-Eb1}
\end{subfigure}

\begin{subfigure}{0.49\textwidth}
	\includegraphics[width=\textwidth]{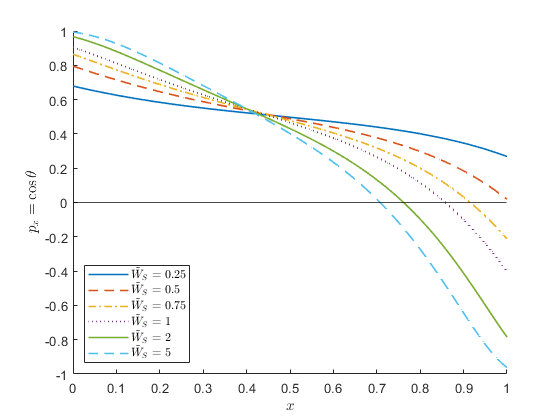}
	\caption{$\tilde{E}_B=0.5$}
	\label{subfig:Ws-Eb2}
\end{subfigure}
\hfill
\begin{subfigure}{0.49\textwidth}
	\includegraphics[width=\textwidth]{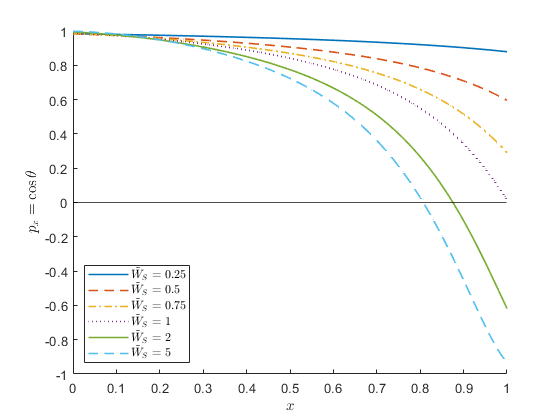}
	\caption{$\tilde{E}_B=1$}
	\label{subfig:Ws-Eb3}
\end{subfigure}

\caption{(a): The boundary values $p_x(0)= \cos\theta(0)$ (solid) and $p_x(1)=\cos\theta(1)$ (dashed) at different applied biases $\E_B$ as a function of the anchoring strength $\W_S$ [note: the $\W_S$-axis begins at $\W_S=0.25$]; (b), (c), (d): The spatial profiles of $p_x=\cos\theta$ at different anchoring strengths and different biases. Parameter values: $\kappa=0.8$, $L/\xi=4$.}
\label{fig:Ws-Eb}
\end{figure}

\begin{figure}[!tb]
\centering

\begin{subfigure}{0.49\textwidth}
	\includegraphics[width=\textwidth]{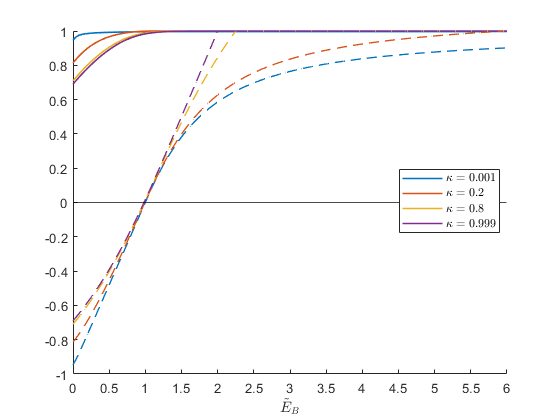}
	\caption{}
	\label{subfig:bdry_Eb-k}
\end{subfigure}
\hfill
\begin{subfigure}{0.49\textwidth}
	\includegraphics[width=\textwidth]{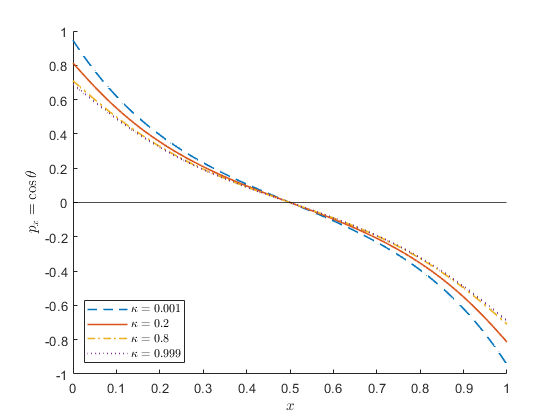}
	\caption{$\tilde{E}_B=0$}
	\label{subfig:Eb-k1}
\end{subfigure}

\begin{subfigure}{0.49\textwidth}
	\includegraphics[width=\textwidth]{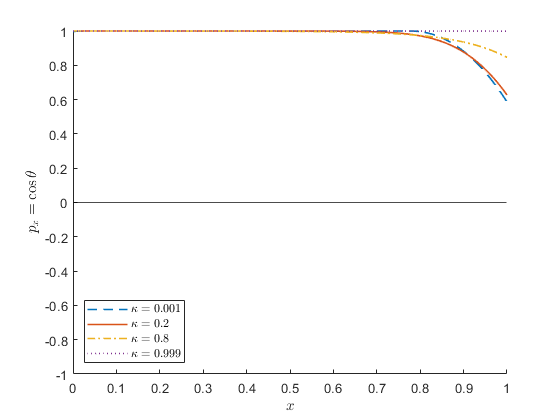}
	\caption{$\tilde{E}_B=2$}
	\label{subfig:Eb-k2}
\end{subfigure}
\hfill
\begin{subfigure}{0.49\textwidth}
	\includegraphics[width=\textwidth]{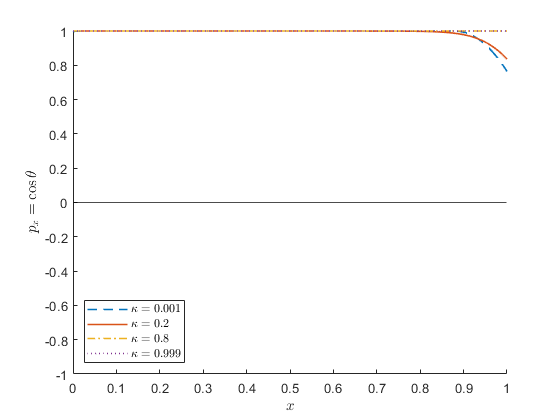}
	\caption{$\tilde{E}_B=3$}
	\label{subfig:Eb-k3}
\end{subfigure}

\caption{(a): The boundary values $p_x(0)= \cos\theta(0)$ (solid) and $p_x(1)=\cos\theta(1)$ (dashed) at different bend-to-splay ratios $\kappa$ as a function of the applied bias $\E_B$; (b), (c), (d): The spatial profiles of $p_x=\cos\theta$ at different ratios and different biases. Parameter values: $\W_S=1$, $L/\xi=4$.}
\label{fig:Eb-k}
\end{figure}

\Cref{fig:Ws-Lxi,fig:Eb-Lxi,fig:Eb-Ws,fig:Ws-Eb,fig:Eb-k} provide insight on several results discussed in \Cref{sec:lproof}. First, since $\theta'>0$ {\color{black} for any nontrivial solution $\theta$} according to \Cref{lm:mono}, it follows that \begin{equation*}
	p_x' = (\cos\theta)' = -(\sin \theta)\theta' < 0 
\end{equation*} for all $\theta\in (0,\pi)$, so that all {\color{black} nonconstant} profiles are monotonically decreasing in $x$. In the case of zero bias, we can also see that the profiles exhibit antisymmetry about the midpoint of the cell $x=\frac{1}{2}$ as discussed in \Cref{lm:sym} (\Cref{fig:Ws-Lxi,subfig:Ws-Eb1,subfig:Eb-k1}). \Cref{fig:Eb-Ws,fig:Ws-Eb} in particular depict the effect of the anchoring strength and applied bias on the {\color{black} polar director configuration}. As expected from {\color{black} \Cref{lm:lxipx0} and} \Cref{lm:Wsasymp} at least in the case of zero bias, increasing the {\color{black} cell thickness $L/\xi$ or} surface anchoring strength $\W_S$ corresponds to an increased preference of the equilibrium profile toward the polar anchoring condition $p_x(0)\approx 1$, $p_x(1)\approx -1$ (\Cref{subfig:bdry_Ws-Lxi,subfig:bdry_Ws-Eb}). On the other hand, {\color{black}  {\color{black} \Cref{fig:Eb-Ws,fig:Ws-Eb,fig:Eb-k} also illustrate that increasing the bias field increasingly dominates the bulk profile and pushes the polar director toward the completely bias-aligned state $p_x\equiv 1$, confirming our conclusion in Lemma \ref{lm:theta1} that $p_x(1)>0$ when $\E_B> \frac{1+\W_S^2}{2}$ and} $p_x\equiv 1$ when $\E_B>1+\frac{\W_S^2}{\kappa}$. However,} {\color{black} the authors were unable to establish a rigorous conclusion regarding the effect of $\E_B$ on $p_x(1)$ {\color{black} in \Cref{lm:Eeff}} because \eqref{eqn:dbd} gives no information about the sign of $\frac{\mathrm{d}p_x(1)}{\mathrm{d}\E_B}$ when $p_x(0)>\E_B$. Still, we believe that $\frac{\mathrm{d}p_x(1)}{\mathrm{d}\E_B}>0$ for all relevant choices of the parameters $\W_S$, $L/\xi$, and $\kappa$ (\Cref{subfig:bdry_Eb-Lxi,subfig:bdry_Eb-Ws,subfig:bdry_Eb-k}), suggesting that the polar director continuously realigns in the direction of the applied field throughout the cell.}

{\color{black} \Cref{fig:Eb-Ws,fig:Ws-Eb,fig:Eb-k} further examine the sharpness of the assumptions in \Cref{lm:bvbigWs,lm:theta1,lm:snd}}. {\color{black} In \Cref{fig:Eb-Ws,fig:Ws-Eb}, we choose $\kappa=0.8$ primarily to illustrate the relevance of $\kappa$ and $\W_S$ in the assumptions of \Cref{lm:theta1,lm:snd}. First, we note that the assumptions in \Cref{lm:bvbigWs,lm:theta1} are not optimal.} Indeed, for $\kappa=0.8$ and $L/\xi=4$, direct calculation shows that \Cref{lm:bvbigWs} requires $\W_S>0.765$ when $\E_B=0.25$, $\W_S>0.839$ when $\E_B=0.5$, and $\W_S>1.061$ when $\E_B=0.75$. Yet, when $\E_B=0.25$ and $\W_S=0.5$, we find \begin{equation*}
	p_x(0)\approx 0.632 > 0.25 > 0 > p_x(1) = -0.222,
\end{equation*} when $\E_B=0.5$ and $\W_S=0.75$, we find \begin{equation*}
	p_x(0)\approx 0.864 > 0.5 > 0 > p_x(1) \approx -0.212,
\end{equation*} and when $\E_B=0.75$ and $\W_S=1$, we find \begin{equation*}
	p_x(0)\approx 0.960 > 0.75 > 0 > p_x(1)\approx -0.206
\end{equation*} (\Cref{subfig:Eb-Ws1,subfig:Eb-Ws2,subfig:Ws-Eb2}). Similarly {\color{black} the $p_x(1)>0$ assertion in} \Cref{lm:theta1} requires $\E_B>0.625$ when $\W_S=0.5$, but $p_x(1)\approx 0.019$ when $\W_S=0.5$ and $\E_B=0.5$ (\Cref{subfig:Eb-Ws1,subfig:Eb-Ws2,subfig:Ws-Eb2}). {\color{black} Yet, \Cref{fig:Eb-k} suggests that both estimates in \Cref{lm:theta1} are nearly optimal. Indeed, \Cref{lm:theta1} implies $p_x(1)>0$ when $\E_B>\frac{1+\W_S^2}{2}$; correspondingly, all $p_x(1)$ curves in \Cref{subfig:bdry_Eb-k} cross the $y$-axis near $\E_B=1$. Moreover, in order to conclude $p_x\equiv 1$ in the case $\W_S=1$, \Cref{lm:theta1} requires $\E_B> 2.001$ when $\kappa=0.999$, $\E_B> 2.25$ when $\kappa=0.8$, $\E_B>6$ when $\kappa=0.2$, all of which are closely supported by \Cref{subfig:bdry_Eb-k}.} {\color{black} \Cref{subfig:Eb-Ws3,subfig:Ws-Eb3} also illustrate the role of the $\W_S$ and $\kappa$ in \Cref{lm:snd}.} Noting that $\frac{1}{2\sqrt{2}} \sqrt{\kappa+2+\sqrt{\kappa^2+8\kappa}}\approx {\color{black} 0.826}$ for $\kappa=0.8$, {when $\E_B=1$, \color{black} \Cref{lm:snd} requires $\W_S<1.101$ in order to guarantee $p_x'' <0$ on $[0,1]$. Indeed, with this choice of parameters, a standard second-order finite-difference approximation shows that $p_x''(1)\approx -15.948$ when $\W_S=1$ and $p_x''(1)\approx 7.948$ when $\W_S=2$, indicating the existence of an inflection point when the surface anchoring is sufficiently strong (\Cref{lm:inflpt})}

\begin{figure}[!tb]
\centering

\begin{subfigure}{0.49\textwidth}
	\includegraphics[width=\textwidth]{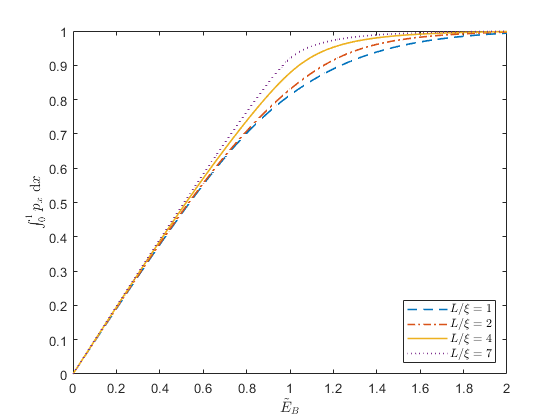}
	\caption{$\W_S=0.5$}
\end{subfigure}
\hfill
\begin{subfigure}{0.49\textwidth}
	\includegraphics[width=\textwidth]{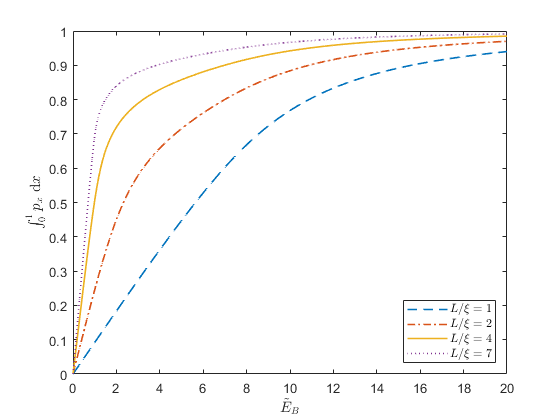}
	\caption{$\W_S=5$}
\end{subfigure}

\begin{subfigure}{0.49\textwidth}
	\includegraphics[width=\textwidth]{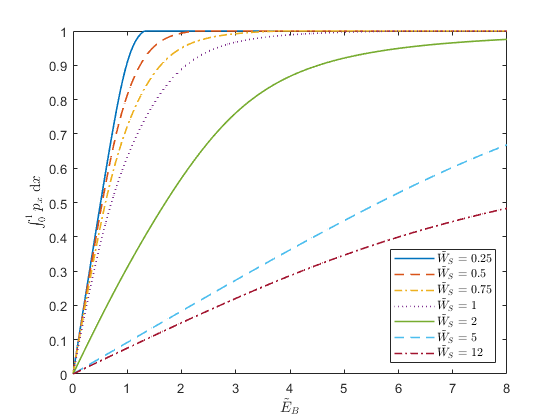}
	\caption{$L/\xi=1$}
\end{subfigure}
\hfill
\begin{subfigure}{0.49\textwidth}
	\includegraphics[width=\textwidth]{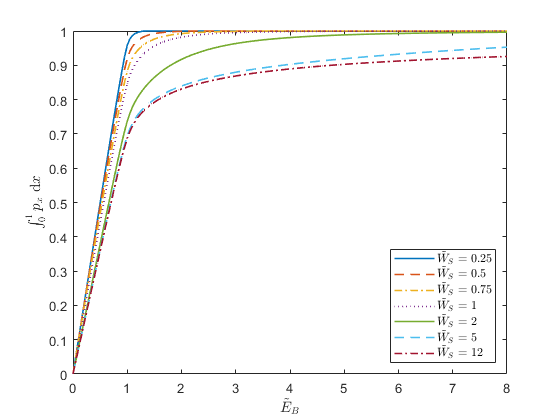}
	\caption{$L/\xi=7$}
\end{subfigure}

\caption{The average value $\langle p_x \rangle$ of the $x$-component of the polar director across the cell as a function of the bias field $\tilde{E}_B$. (a), (b): The average value at different cell thicknesses $L/\xi$; (c), (d): The average value at different anchorings $\tilde{W}_S$. In all cases $\kappa=0.2$.}
\label{fig:avval}
\end{figure}

\Cref{fig:avval} shows the average value $\langle p_x \rangle$ of the $x$-component of the equilibrium polar director as a function of the applied bias at different cell thicknesses and surface anchoring strengths. As anticipated from \Cref{lm:Eeff}, in all cases we observe that $\langle p_x \rangle \to 1$ as $\E_B\to +\infty$, illustrating the expected realignment of the molecules toward the applied bias as the bias becomes dominant for energy-minimizing configurations.

\begin{figure}[!tb]
\centering

\begin{subfigure}{0.49\textwidth}
	\includegraphics[width=\textwidth]{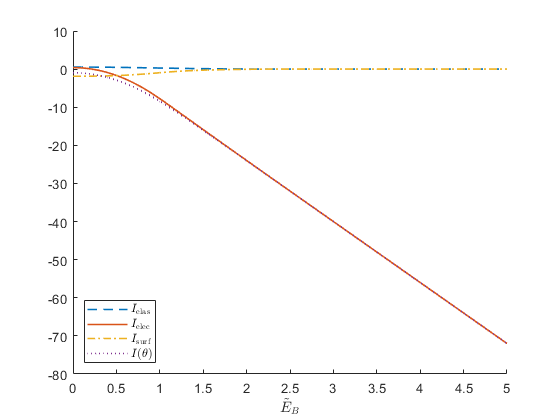}
	\caption{$\tilde{W}_S=0.5$, $L/\xi=4$}
\end{subfigure}
\hfill
\begin{subfigure}{0.49\textwidth}
	\includegraphics[width=\textwidth]{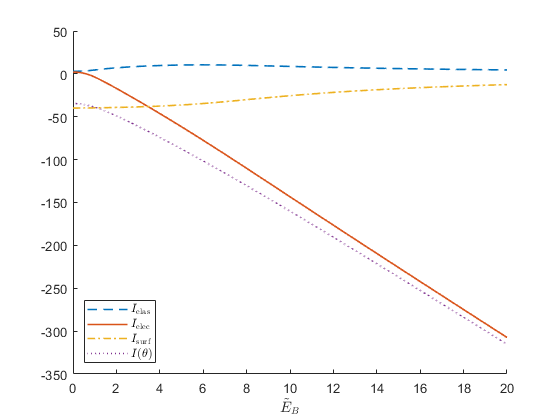}
	\caption{$\tilde{W}_S=5$, $L/\xi=4$}
\end{subfigure}

\begin{subfigure}{0.49\textwidth}
	\includegraphics[width=\textwidth]{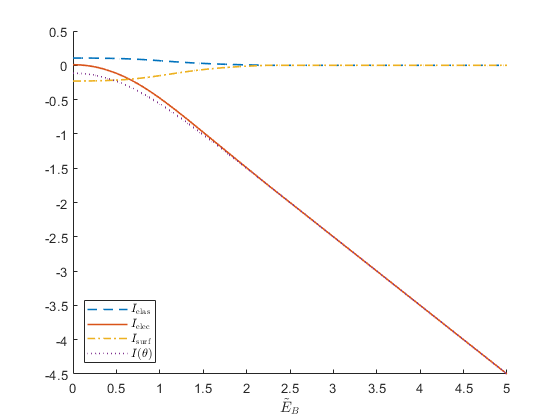}
	\caption{$\tilde{W}_S=0.5$, $L/\xi=1$}
\end{subfigure}
\hfill
\begin{subfigure}{0.49\textwidth}
	\includegraphics[width=\textwidth]{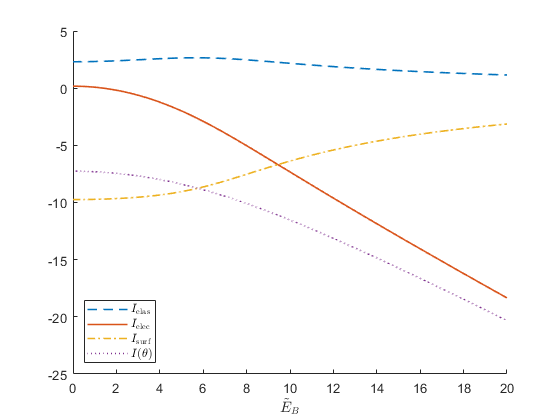}
	\caption{$\tilde{W}_S=5$, $L/\xi=1$}
\end{subfigure}

\caption{Decomposition of the energy $I(\theta)$ into elastic ($I_\text{elas}$), electrostatic ($I_\text{elec}$), and surface ($I_\text{surf}$) contributions as a function of the bias field $\tilde{E}_B$. In all cases $\kappa=0.2$.}
\label{fig:Eb-decomp}
\end{figure}

The {\color{black} above observations} suggest an underlying physical competition between the cell thickness, surface anchoring strength, and applied bias. First, we can see in \Cref{subfig:Eb-Lxi1,subfig:Ws-Lxi1} that the $x$-component of the equilibrium polar director appears linear for the small thickness $L/\xi=1$ and then exhibits a greater spatial variation with sharp transitions at the boundaries at the higher thickness $L/\xi=7$. This behavior again suggests a transition from elastic-driven boundary effects to more electrostatic-driven bulk effects as $L/\xi$ increases. The other primary competition occurs between the surface anchoring strength $\W_S$ and the applied bias $\E_B$, as depicted in the spatial profiles in \Cref{fig:Eb-Ws,fig:Ws-Eb,fig:Eb-k}. At the higher anchoring strength $\W_S=5$, the equilibrium polar director shows a clear preference toward the polar anchoring condition $p_x(0)\approx 1$, $p_x(1)\approx -1$ in {\color{black} each of} the cases of zero, moderate, and strong bias (\Cref{subfig:Ws-Eb1,subfig:Ws-Eb2,subfig:Ws-Eb3}). Yet, we still see that increasing the applied bias appears to lead to a stronger preference toward the electrostatically favored state $p_x\equiv 1$ in both bulk alignment (\Cref{subfig:Ws-Eb1,subfig:Ws-Eb2,subfig:Ws-Eb3}) and average value (\Cref{fig:avval}). This again indicates that the equilibrium polar director configuration is driven primarily by the bulk electrostatic contributions at high bias. 

\begin{figure}[!tb]
\centering

\begin{subfigure}{0.49\textwidth}
	\includegraphics[width=\textwidth]{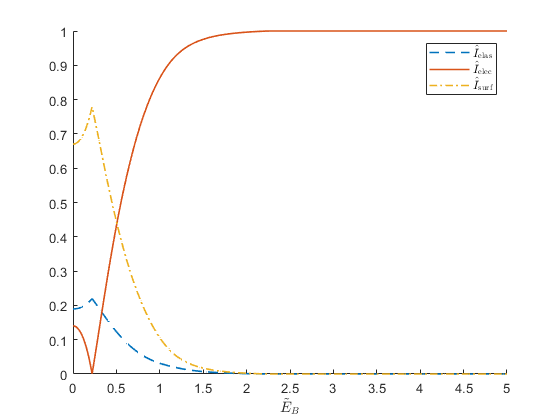}
	\caption{$\tilde{W}_S=0.5$, $L/\xi=4$}
	\label{subfig:norma}
\end{subfigure}
\hfill
\begin{subfigure}{0.49\textwidth}
	\includegraphics[width=\textwidth]{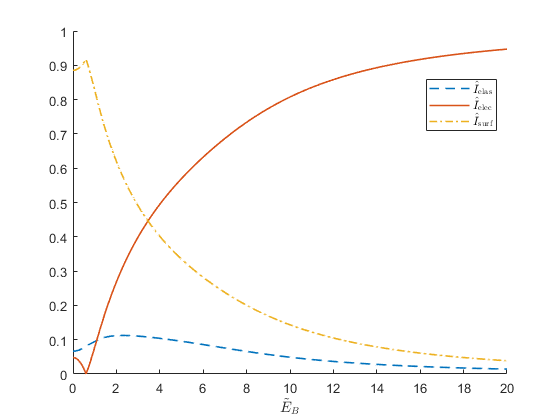}
	\caption{$\tilde{W}_S=5$, $L/\xi=4$}
	\label{subfig:normb}
\end{subfigure}

\begin{subfigure}{0.49\textwidth}
	\includegraphics[width=\textwidth]{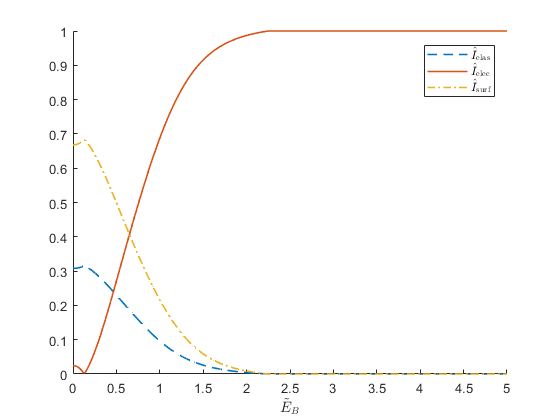}
	\caption{$\tilde{W}_S=0.5$, $L/\xi=1$}
	\label{subfig:normc}
\end{subfigure}
\hfill
\begin{subfigure}{0.49\textwidth}
	\includegraphics[width=\textwidth]{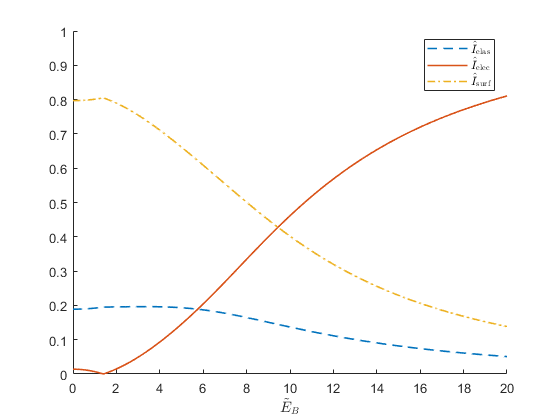}
	\caption{$\tilde{W}_S=5$, $L/\xi=1$}
	\label{subfig:normd}
\end{subfigure}

\caption{The normalized contributions $\hat{I}_\text{elas}$, $\hat{I}_\text{elec}$, and $\hat{I}_\text{surf}$ as a function of the bias field $\tilde{E}_B$. (a) and (b): Moderate thickness $L/\xi=4$, (c) and (d): Thin cells $L/\xi = 1$. In all cases $\kappa=0.2$.}
\label{fig:Eb-norm-decomp}
\end{figure}

Motivated by this apparent transition from boundary elastic to bulk electrostatic dominating effects (\Cref{fig:Eb-Ws,fig:Ws-Eb,fig:Eb-k,fig:avval}), we consider the energetic mechanisms underlying equilibrium polar director structure at high bias. Specifically, we consider the energy $I(\theta)$ as defined in \eqref{eqn:enI} and decompose $I(\theta)$ into the sum of elastic $I_\text{elas}$, electrostatic $I_\text{elec}$, and surface $I_\text{surf}$ contributions defined as follows:  \begin{equation}\label{eqn:Ielas}
	I_\text{elas} = \int_0^1 \frac{1}{2}(\sin^2\theta+\kappa\cos^2\theta)(\theta^\prime)^2 dx, 
\end{equation} \begin{equation}\label{eqn:Ielec}
	I_\text{elec} = \int_0^1 \frac{1}{2}\left(\frac{L}{\xi}\right)^2(\cos^2\theta-2\tilde{E}_B\cos\theta) dx,
\end{equation} and \begin{equation}\label{eqn:Isurf}
	I_\text{surf} = -\W_S\left(\frac{L}{\xi}\right)\cos\theta(0) + \W_S\left(\frac{L}{\xi}\right)\cos\theta(1).
\end{equation} Notice also that \begin{equation*}
	I(\theta) = I_\mathrm{elas}+ I_\mathrm{elec}+ I_\mathrm{surf}.
\end{equation*} We then define the normalized energy contributions $\hat{I}_\text{elas}$, $\hat{I}_\text{elec}$, and $\hat{I}_\text{surf}$ by \begin{equation}\label{eqn:contrib}
	\hat{I}_\text{elas} = \frac{|I_\text{elas}|}{I_\text{mag}},\quad \hat{I}_\text{elec} = \frac{|I_\text{elec}|}{I_\text{mag}},\quad \hat{I}_\text{surf} = \frac{|I_\text{surf}|}{I_\text{mag}},
\end{equation} where $I_\text{mag}= |I_\text{elas}|+|I_\text{elec}|+|I_\text{surf}|$.

\Cref{fig:Eb-decomp} shows a plot of the contributions $I_\text{elas}$, $I_\text{elec}$, and $I_\text{surf}$ along with the total energy $I(\theta)$ as a function of the bias field $\E_B$, and \Cref{fig:Eb-norm-decomp} shows a plot of the normalized contributions $\hat{I}_\text{elas}$, $\hat{I}_\text{elec}$, and $\hat{I}_\text{surf}$. For small bias, the surface term dominates, indicating that the configuration of the polar director is driven primarily by boundary effects. However, as the bias increases, a transition occurs as the electrostatic term becomes dominant. For sufficiently large bias $\E_B$, the energy is almost entirely determined by the electrostatic term, and the elastic and surface contributions become negligible (\Cref{fig:Eb-norm-decomp}). We note that this transition to an electrostatically-dominated energy appears to require higher values of the applied bias $\E_B$ at larger anchorings $\W_S>\!>1$ and for thinner cells $L/\xi \approx 1$. This behavior can be interpreted as a requirement that the applied bias must be large enough to overcome greater surface and elastic contributions in thin cells ($L/\xi\approx 1$) with large anchoring.

This transition toward electrostatically-dominating energetics can be understood variationally in the limit as $\E_B\to +\infty$.

\begin{proposition}\label{p:var}
	Let $\theta\in W^{1,p}(0,1)$ be such that $I(\theta)\le I(0)$, where $I$ is as defined in \eqref{eqn:enI}. Then there exists a constant $C=C(\W_S,\frac{L}{\xi})>0$, independent of $\E_B$, such that \begin{equation}\label{eqn:varineq}
		I_\text{elas} + \E_B\left(\frac{L}{\xi}\right)^2 \int_0^1 (1-\cos\theta) dx \le C.
	\end{equation} In particular, as $\E_B\to +\infty$, we have \begin{equation}\label{eqn:varpx}
		\langle p_x \rangle \to 1
	\end{equation} and \begin{equation}\label{eqn:varint}
		\hat{I}_\text{elec}\to 1,
	\end{equation} where $\langle p_x \rangle$ is as defined in \eqref{eqn:avval} and $\hat{I}_\text{elec}$ is as defined in \eqref{eqn:contrib}.
\end{proposition}

\begin{proof}
	First, notice that \begin{equation}\label{eqn:surfbd}
		I_\text{surf} = \big(\cos\theta(1)-\cos\theta(0)\big)\W_S\left(\frac{L}{\xi}\right) = \alpha\W_S\left(\frac{L}{\xi}\right)
	\end{equation} for some $\alpha\in \R$ such that $|\alpha|\le 2$. 
	
	Next, observe that, for any director $\theta$ such that $I(\theta)\le I(0)$, we have \begin{equation}\label{eqn:Icomp}
	\begin{split}
		I(\theta) &\le I(0) = \frac{1}{2}\left(\frac{L}{\xi}\right)^2(1-2\E_B).
	\end{split}
	\end{equation} Expanding $I(\theta)$ in terms of $I_\text{elas}$, $I_\text{elec}$, and $I_\text{surf}$ gives \begin{equation}\label{}
	\begin{split}
		I(\theta) &= I_\text{elas} + I_\text{elec} + I_\text{surf} \\
			&= I_\text{elas} + \frac{1}{2}\left(\frac{L}{\xi}\right)^2\int_0^1 \big(\cos^2\theta - 2\E_B\cos\theta\big) dx + \alpha\W_S\left(\frac{L}{\xi}\right)\\
			&= I_\text{elas} -\E_B\left(\frac{L}{\xi}\right)^2\int_0^1 \cos\theta dx + \frac{\beta}{2}\left(\frac{L}{\xi}\right)^2 + \alpha\W_S\left(\frac{L}{\xi}\right),
	\end{split}
	\end{equation} where $\beta=\int_0^1 \cos^2\theta dx\in [0,1]$. Inserting this back into \eqref{eqn:Icomp} and rearranging shows \begin{equation*}
		I_\text{elas}+ \E_B\left(\frac{L}{\xi}\right)^2\int (1-\cos\theta)dx \le \frac{1}{2}\left(\frac{L}{\xi}\right)^2(1-\beta) -\alpha\W_S\left(\frac{L}{\xi}\right).
	\end{equation*} Since $I_\mathrm{elas}\ge 0$ and $|\alpha|,|\beta|<+\infty$, this proves \eqref{eqn:varineq}.
	
	To show \eqref{eqn:varpx}, since $I_\text{elas}\ge 0$, \eqref{eqn:varineq} implies \begin{equation*}
		\int_0^1 (1-\cos\theta)d x \le \frac{C\left(\frac{\xi}{L}\right)^2}{\E_B}
	\end{equation*} for some constant $C>0$ independent of $\E_B$. Taking the limit as $\E_B\to +\infty$ on both sides proves \eqref{eqn:varpx}.
	
	Lastly, \eqref{eqn:surfbd} and \eqref{eqn:varineq} together imply that \begin{equation*}
		|I_\text{elas}|+|I_\text{surf}|\le \tilde{C}
	\end{equation*} for some constant $\tilde{C}>0$ independent of $\E_B$. On the other hand, \eqref{eqn:varpx} implies that \begin{equation*}
		\lim_{\E_B\to +\infty} \frac{1}{2}\left(\frac{L}{\xi}\right)^2\int_0^1 (\cos^2\theta - 2\E_B\cos \theta) dx = \frac{\beta}{2}\left(\frac{L}{\xi}\right)^2 - \lim_{\E_B\to +\infty} \E_B\int_0^1 \cos\theta d x = -\infty,
	\end{equation*} where again $0\le \beta \le 1$. It then follows \begin{equation*}
		\lim_{\E_B\to +\infty} \frac{|I_\text{elec}|}{|I_\text{elas}|+|I_\text{elec}|+|I_\text{surf}|}=1.
	\end{equation*} This proves \eqref{eqn:varint}. The proof is complete.
\end{proof}

Together, \Cref{p:var} and \Cref{fig:Eb-decomp,fig:Eb-norm-decomp} suggest that the effects of the applied bias field $\E_B$ discussed in \Cref{lm:Eeff} are driven primarily by this transition from boundary-dominated to electrostatically-dominated energetic contributions with increasing bias. Indeed the dominating electrostatic effects are reflected in an alignment of the equilibrium polar director with the field both in bulk profile (\Cref{fig:Eb-Ws,fig:Ws-Eb}) and average value (\Cref{fig:avval}). The associated stability properties of these equilibria, particularly at high bias, remain an interesting open direction for future work.

\subsection*{Acknowledgments} We thank Tiziana Giorgi  for introducing us to bent-core liquid crystals and for helpful discussions on the results of \cite{GGJ23}.  

\subsection*{Funding Declaration:} A.D.W. is partially supported by NSF grant DMS-2154047. X.Y.  is partially supported by Simons travel grant 947054 and NSF grant DMS-2306393.

\subsection*{Declaration of generative AI in the writing process}
During the proof correction process upon article acceptance at Journal of Mathematical Analysis, the authors used OpenAI's ChatGPT to assist with proofreading, identifying potential inconsistencies in mathematical statements and arguments, and checking bibliographic information. In particular, clarifications were made to the statements of the existence and uniqueness theorem \Cref{thm:existence}, Part (3) of \Cref{thm:bdry} and the associated \Cref{lm:theta1}, \Cref{lm:Eeff}, and the corresponding remarks. The authors independently reviewed and evaluated the suggested corrections, made all final editorial and mathematical decisions, and take full responsibility for the content of the manuscript.

\clearpage
\bibliographystyle{elsarticle-num} 
\bibliography{references.bib}

@article{Baietal07,
  title = {Structure and stability of bent core liquid crystal fibers},
  author = {Bailey, C. and Gartland, E. C. and J\'akli, A.},
  journal = {Phys. Rev. E},
  volume = {75},
  issue = {3},
  pages = {031701},
  numpages = {9},
  year = {2007},
  month = {Mar},
  publisher = {American Physical Society},
  doi = {10.1103/PhysRevE.75.031701},
  url = {https://link.aps.org/doi/10.1103/PhysRevE.75.031701}
}

@article {BP12,
    AUTHOR = {Bauman, Patricia and Phillips, Daniel},
     TITLE = {Analysis and stability of bent-core liquid crystal fibers},
   JOURNAL = {Discrete Contin. Dyn. Syst. Ser. B},
  FJOURNAL = {Discrete and Continuous Dynamical Systems. Series B. A Journal
              Bridging Mathematics and Sciences},
    VOLUME = {17},
      YEAR = {2012},
    NUMBER = {6},
     PAGES = {1707--1728},
      ISSN = {1531-3492,1553-524X},
   MRCLASS = {76A15 (35R35 82D30)},
  MRNUMBER = {2924435},
MRREVIEWER = {Giovanni\ De Matteis},
       DOI = {10.3934/dcdsb.2012.17.1707},
       URL = {https://doi.org/10.3934/dcdsb.2012.17.1707},
}

@article{BCP92,
author = {Brand, Helmut R. and Cladis, P. E. and Pleiner, Harald},
title = {Symmetry and defects in the CM phase of polymeric liquid crystals},
journal = {Macromolecules},
volume = {25},
number = {26},
pages = {7223-7226},
year = {1992},
doi = {10.1021/ma00052a025},

URL = { 
    
        https://doi.org/10.1021/ma00052a025
    
    

},
eprint = { 
    
        https://doi.org/10.1021/ma00052a025
    
    

}

}

@article {BCP98,
    AUTHOR = {Brand, H. R. and  Cladis, P. E. and Pleiner, H.},
     TITLE = {Macroscopic properties of smectic liquid crystals},
   JOURNAL = {Eur. Phys. J. B},
    VOLUME = {6},
      YEAR = {1998},
     PAGES = {347-353},
       DOI = {10.1007/s100510050560},
       URL = {https://doi.org/10.1007/s100510050560},
}

@ARTICLE{CPB99,
       author = {{E. Cladis} and  H. Pleiner and  H.~R. Brand},
        title = "{Fluid Biaxial Banana Smectics: Symmetry at Work}",
      journal = {Liquid Crystals Today},
         year = 1999,
        month = dec,
       volume = {9},
       number = {3-4},
        pages = {1-6},
          doi = {10.1080/13583149950076487},
       adsurl = {https://ui.adsabs.harvard.edu/abs/1999LCTod...9....1E}
}

@article {GGJ20,
    AUTHOR = {Garc\'ia-Cervera, Carlos J. and Giorgi, Tiziana and Joo,
              Sookyung},
     TITLE = {Boundary vortex formation in polarization-modulated orthogonal
              smectic liquid crystals},
   JOURNAL = {SIAM J. Appl. Math.},
  FJOURNAL = {SIAM Journal on Applied Mathematics},
    VOLUME = {80},
      YEAR = {2020},
    NUMBER = {5},
     PAGES = {2024--2044},
      ISSN = {0036-1399,1095-712X},
   MRCLASS = {76A15 (35B40 35J50 35J57 35Q35 35Q56 82D30)},
  MRNUMBER = {4146771},
MRREVIEWER = {Piotr\ Biler},
       DOI = {10.1137/19M1301618},
       URL = {https://doi.org/10.1137/19M1301618},
}

@article{GGJ23,
    AUTHOR = {Garc\'ia-Cervera, Carlos J. and Giorgi, Tiziana and Joo,
              Sookyung},
     TITLE = {Dimensional reduction for the ferroelectric smectic {A}-type
              phase of bent-core liquid crystals},
   JOURNAL = {J. Nonlinear Sci.},
  FJOURNAL = {Journal of Nonlinear Science},
    VOLUME = {33},
      YEAR = {2023},
    NUMBER = {1},
     PAGES = {Paper No. 19},
      ISSN = {0938-8974,1432-1467},
   MRCLASS = {35J08 (76A15)},
  MRNUMBER = {4519638},
       DOI = {10.1007/s00332-022-09874-x},
       URL = {https://doi.org/10.1007/s00332-022-09874-x},
}

@article{GY15,
title = {Analysis of a model for bent-core liquid crystals columnar phases},
journal = {Discrete and Continuous Dynamical Systems - B},
volume = {20},
number = {7},
pages = {2001-2026},
year = {2015},
issn = {1531-3492},
doi = {10.3934/dcdsb.2015.20.2001},
url = {https://www.aimsciences.org/article/id/d2a236d1-cc68-40fa-af5b-ff8bb7dc15a2},
author = {Tiziana Giorgi and Feras Yousef}
}

@article{GCV14,
  title = {Effect of a bias electric field on the structure and dielectric response of the ferroelectric smectic-$A$ liquid crystal in thin planar cells},
  author = {Gornik, K. and \ifmmode \check{C}\else \v{C}\fi{}epi\ifmmode \check{c}\else \v{c}\fi{}, M. and Vaupoti\ifmmode \check{c}\else \v{c}\fi{}, N.},
  journal = {Phys. Rev. E},
  volume = {89},
  issue = {1},
  pages = {012501},
  numpages = {9},
  year = {2014},
  month = {Jan},
  publisher = {American Physical Society},
  doi = {10.1103/PhysRevE.89.012501},
  url = {https://link.aps.org/doi/10.1103/PhysRevE.89.012501}
}

@article{Guoetal11,
  title = {Ferroelectric behavior of orthogonal smectic phase made of bent-core molecules},
  author = {Guo, Lingfeng and Gorecka, Ewa and Pociecha, Damian and Vaupoti\ifmmode \check{c}\else \v{c}\fi{}, Nata\ifmmode \check{s}\else \v{s}\fi{}a and \ifmmode \check{C}\else \v{C}\fi{}epi\ifmmode \check{c}\else \v{c}\fi{}, Mojca and Reddy, R. Amaranatha and Gornik, Kristina and Araoka, Fumito and Clark, Noel A. and Walba, David M. and Ishikawa, Ken and Takezoe, Hideo},
  journal = {Phys. Rev. E},
  volume = {84},
  issue = {3},
  pages = {031706},
  numpages = {8},
  year = {2011},
  month = {Sep},
  publisher = {American Physical Society},
  doi = {10.1103/PhysRevE.84.031706},
  url = {https://link.aps.org/doi/10.1103/PhysRevE.84.031706}
}

@article{JLOS18,
  title = {Physics of liquid crystals of bent-shaped molecules},
  author = {J\'akli, Antal and Lavrentovich, Oleg D. and Selinger, Jonathan V.},
  journal = {Rev. Mod. Phys.},
  volume = {90},
  issue = {4},
  pages = {045004},
  numpages = {68},
  year = {2018},
  month = {Nov},
  publisher = {American Physical Society},
  doi = {10.1103/RevModPhys.90.045004},
  url = {https://link.aps.org/doi/10.1103/RevModPhys.90.045004}
}

@article{
Lin97,
author = {Darren R. Link  and Giorgio Natale  and Renfan Shao  and Joseph E. Maclennan  and Noel A. Clark  and Eva Körblova  and David M. Walba },
title = {Spontaneous Formation of Macroscopic Chiral Domains in a Fluid Smectic Phase of Achiral Molecules},
journal = {Science},
volume = {278},
number = {5345},
pages = {1924-1927},
year = {1997},
doi = {10.1126/science.278.5345.1924},
URL = {https://www.science.org/doi/abs/10.1126/science.278.5345.1924},
eprint = {https://www.science.org/doi/pdf/10.1126/science.278.5345.1924}}

@article{Meyetal75,
  TITLE = {{Ferroelectric liquid crystals}},
  AUTHOR = {Meyer, R.B. and Liebert, L. and Strzelecki, L. and Keller, P.},
  URL = {https://hal.science/jpa-00231156},
  JOURNAL = {{Journal de Physique Lettres}},
  PUBLISHER = {{Edp sciences}},
  VOLUME = {36},
  NUMBER = {3},
  PAGES = {69-71},
  YEAR = {1975},
  DOI = {10.1051/jphyslet:0197500360306900},
  HAL_ID = {jpa-00231156},
  HAL_VERSION = {v1},
}

@Article{Nioetal96,
author ="Niori, T. and Sekine, T. and Watanabe, J. and Furukawa, T. and Takezoe, H.",
title  ="Distinct ferroelectric smectic liquid crystals consisting of banana shaped achiral molecules",
journal  ="J. Mater. Chem.",
year  ="1996",
volume  ="6",
issue  ="7",
pages  ="1231-1233",
publisher  ="The Royal Society of Chemistry",
doi  ="10.1039/JM9960601231",
url  ="http://dx.doi.org/10.1039/JM9960601231"}

@article{OP14,
  title = {Molecular theory of proper ferroelectricity in bent-core liquid crystals},
  author = {Osipov, M. A. and G Pajak},
  journal = {The European physical journal. E, Soft matter},
  volume = {37},
  issue = {9},
  numpages = {7},
  year = {2014},
  doi = {10.1140/epje/i2014-14079-0}
}

@ARTICLE{Pel99,
       author = {{Pelzl}, Gerhard and {Diele}, Siegmar and {Weissflog}, Wolfgang},
        title = "{Banana-Shaped Compounds{\textemdash}A New Field of Liquid Crystals}",
      journal = {Advanced Materials},
         year = 1999,
        month = jun,
       volume = {11},
       number = {9},
        pages = {707-724},
          doi = {10.1002/(SICI)1521-4095(199906)11:9<707::AID-ADMA707>3.0.CO;2-D},
       adsurl = {https://ui.adsabs.harvard.edu/abs/1999AdM....11..707P}
}

@article{Vor29,
author = {Vorländer, D.},
title = {Die Richtung der Kohlenstoff-Valenzen in Benzol-Abkömmlingen},
journal = {Berichte der deutschen chemischen Gesellschaft (A and B Series)},
volume = {62},
number = {10},
pages = {2831-2835},
doi = {https://doi.org/10.1002/cber.19290621026},
url = {https://chemistry-europe.onlinelibrary.wiley.com/doi/abs/10.1002/cber.19290621026},
eprint = {https://chemistry-europe.onlinelibrary.wiley.com/doi/pdf/10.1002/cber.19290621026},
year = {1929}
}

@article{Sek97,
doi = {10.1143/JJAP.36.L1201},
url = {https://dx.doi.org/10.1143/JJAP.36.L1201},
year = {1997},
month = {sep},
publisher = {},
volume = {36},
number = {9A},
pages = {L1201},
author = {Sekine, Tomoko and Takanishi, Yoichi and Niori, Teruki and Takezoe, Junji Watanabe},
title = {Ferroelectric Properties in Banana-Shaped Achiral Liquid Crystalline  Molecular Systems},
journal = {Japanese Journal of Applied Physics}
}

@article{Red11,
  title = {Spontaneous ferroelectric order in a bent-core smectic liquid crystal of fluid orthorhombic layers},
  author = {R. Amaranatha Reddy  and Chenhui Zhu  and Renfan Shao  and Eva Korblova  and Tao Gong  and Yongqiang Shen  and Edgardo Garcia  and Matthew A. Glaser  and Joseph E. Maclennan  and David M. Walba  and Noel A. Clark },
  journal = {Science},
  volume = {332},
  issue = {6025},
page={72-77},
  numpages = {5},
  year = {2011},
  doi = {10.1126/science.1197248}
}

@article{TT06,
doi = {10.1143/JJAP.45.597},
url = {https://dx.doi.org/10.1143/JJAP.45.597},
year = {2006},
month = {feb},
publisher = {},
volume = {45},
number = {2R},
pages = {597},
author = {Takezoe, Hideo and Takanishi, Yoichi},
title = {Bent-Core Liquid Crystals: Their Mysterious and Attractive World},
journal = {Japanese Journal of Applied Physics}
}

@article{VA32,
author = {Vorländer, D. and Apel, Alfons},
title = {Die Richtung der Kohlenstoff-Valenzen in Benzolabkömmlingen (II.)},
journal = {Berichte der deutschen chemischen Gesellschaft (A and B Series)},
volume = {65},
number = {7},
pages = {1101-1109},
doi = {https://doi.org/10.1002/cber.19320650710},
url = {https://chemistry-europe.onlinelibrary.wiley.com/doi/abs/10.1002/cber.19320650710},
eprint = {https://chemistry-europe.onlinelibrary.wiley.com/doi/pdf/10.1002/cber.19320650710},
year = {1932}
}

@Article{RT06,
author ="Reddy, R. Amaranatha and Tschierske, Carsten",
title  ="Bent-core liquid crystals: polar order{,} superstructural chirality and spontaneous desymmetrisation in soft matter systems",
journal  ="J. Mater. Chem.",
year  ="2006",
volume  ="16",
issue  ="10",
pages  ="907-961",
publisher  ="The Royal Society of Chemistry",
doi  ="10.1039/B504400F",
url  ="http://dx.doi.org/10.1039/B504400F"}

@article{
Wal00,
author = {David M. Walba  and Eva Körblova  and Renfan Shao  and Joseph E. Maclennan  and Darren R. Link  and Matthew A. Glaser  and Noel A. Clark },
title = {A Ferroelectric Liquid Crystal Conglomerate Composed of Racemic Molecules},
journal = {Science},
volume = {288},
number = {5474},
pages = {2181-2184},
year = {2000},
doi = {10.1126/science.288.5474.2181},
URL = {https://www.science.org/doi/abs/10.1126/science.288.5474.2181},
eprint = {https://www.science.org/doi/pdf/10.1126/science.288.5474.2181}}

@incollection{Wro11,
author = {Stanisław Wr\'obel and Janusz Chru\'sciel and Marta Wierzejska-Adamowicz and Monika Marzec and Danuta M. Ossowska-Chru\'sciel and Christian Legrand and Redouane Douali},
title = {Ferroelectric Liquid Crystals Composed of Banana-Shaped Thioesters},
booktitle = {Ferroelectrics},
publisher = {IntechOpen},
address = {Rijeka},
year = {2011},
editor = {Micka\"el Lallart},
chapter = {18},
doi = {10.5772/22155},
url = {https://doi.org/10.5772/22155}
}

@article{MasM1991,
author = {Masuda, Tatsuya and Matsunaga, Yoshio},
title = {Mesomorphic Properties of 4-[3,4-Bis(4-alkoxybenzoyloxy)benzylideneamino]azobenzenes and Related Compounds},
journal = {Bulletin of the Chemical Society of Japan},
volume = {64},
number = {7},
pages = {2192-2195},
year = {1991},
publisher = {Taylor \& Francis},
doi = {10.1246/bcsj.64.2192},


URL = { 
    
        https://doi.org/10.1246/bcsj.64.2192
    
    

},
eprint = { 
    
        https://doi.org/10.1080/02678299308027306
    
    

}

}

@article{MM1993,
author = {Hiroyuki Matsuzaki and Yoshio Matsunaga},
title = {New mesogenic compounds with unconventional molecular structures 1,2-Phenylene and 2,3-naphthylene bis[4-(4-alkoxyphenyliminomethyl)benzoates] and related compounds},
journal = {Liquid Crystals},
volume = {14},
number = {1},
pages = {105--120},
year = {1993},
publisher = {Taylor \& Francis},
doi = {10.1080/02678299308027306},


URL = { 
    
        https://doi.org/10.1080/02678299308027306
    
    

},
eprint = { 
    
        https://doi.org/10.1080/02678299308027306
    
    

}

}

@article{KMM1991,
author = {M. Kuboshita and Y. Matsunaga and H. Matsuzaki},
title = {Mesomorphic Behavior of 1,2-Phenylene Bis[4-(4-alkoxybenzylideneamino)benzoates},
journal = {Molecular Crystals and Liquid Crystals},
volume = {199},
number = {1},
pages = {319--326},
year = {1991},
publisher = {Taylor \& Francis},
doi = {10.1080/00268949108030943},


URL = { 
    
        https://doi.org/10.1080/00268949108030943
    
    

},
eprint = { 
    
        https://doi.org/10.1080/00268949108030943
    
    

}

}

@phdthesis{Les16, place={Maribor}, title={Ferroelectric smectic-A phase made of bent-core liquid crystals: Structure and Dielectric response in thin cells}, url={https://dk.um.si/IzpisGradiva.php?lang=eng&id=57326}, abstractNote={In the thesis we study the structure and response of bent-core liquid crystals in the orthogonal ferroelectric Smectic-A (SmAP$_F$) phase in thin planar cells. We construct a phenomenological continuum model to study the structure in thin planar cells. A set of molecules within a small volume is presented by the director ($vec{n}$), which defines the average direction of the long molecular axes within this volume, and by the polar director ($vec{p}$), which points in the direction of local polarization. We choose a geometry in which the director in the smectic layer is constant and the polar director varies across the cell. The polar director structure inside the cell is determined by a competition among the torques due to the bulk elasticity, electrostatic effects and surface anchoring. The equilibrium profile structure of the polar director is obtained by minimization of the free energy. We find the polar director profile in a cell as a function of the type and strength of the surface anchoring, bulk elastic constants and cell thickness.    The effect of the external electric bias field on the structure in the cell is studied, as well. Bent-core molecules have a permanent electric dipole moment. In the external electric field the electric torque tends to rotate the dipoles in the direction of the external field. The polar director profile in the cell in external bias field thus depends on the competition among three effects: anchoring at the surfaces, the elastic properties of the bulk and the influence of the external electric field. By the rotation of molecules in the external electric field, optical properties of the cell are changed, therefore the SmAP$_F$ phase is a promising phase for use in displays with high response time, high contrast, continuous gray level and wide viewing angle.    The response of the SmAP$_F$ phase to an alternating external electric field (the dielectric response) is also considered. The dielectric response of the SmAP$_F$ phase consists of two modes: the phase and amplitude mode. The phase mode is due to fluctuations in the orientation of the local direction of the spontaneous polarization and the amplitude mode is due to the change in the magnitude of spontaneous polarization. The frequency of the phase and amplitude mode and the dielectric permittivity are calculated numerically as a function of the bias external DC electric field, cell thickness, the type and strength of surface anchoring and the ratio between the bend and splay elastic constants. Analytical solution for a very specific case of chosen parameters is also obtained. Theoretically obtained dependencies are in agreement with the reported experimental measurements.    In the thesis we study the effect of different types of surface anchoring and for this purpose three different types of cells are defined. The polar director profile and the dielectric response is calculated in all three types of cells. The type I cell has polar surface anchoring of equal strengths at both surfaces. The type II cell has, in addition to the polar surface anchoring of equal strengths at both surfaces, a nonpolar anchoring at the bottom surface. The type III cell has polar surface anchoring at both surfaces but not of equal strengths. We predict that by comparison of the dielectric response of the SmAP$_F$ phase in all three types of cells the type of the polarization splay in bent-core liquid crystals can be determined.    The thesis is divided into three parts. In the first part the basic physical properties of ferroelectric liquid crystals, focusing on the bent-core liquid crystals, are discussed. In the second part phenomenological theoretical model is developed. In the third part of the thesis a dielectric response in external bias field is studied.}, school={University of Maritoba}, author={Leskovar, Kristina}, year={2016}}

@article{Zhu12,
  title = {Topological ferroelectric bistability in a polarization-modulated orthogonal smectic liquid crystal.},
  author = {Zhu, Chenhui and Shao, Renfan and Reddy, R. Amaranatha and Chen, Dong and Shen, Yongqiang and Gong, Tao and Glaser, Matthew
  A. and Korblova, Eva and Rudquist, Per and Maclennan, Joseph E. and Walba, David M. and Clark, Noel A.},
  journal = {Journal of the American Chemical Society },
volume = {134},
number = {23},
pages = {9681-9687},
  year = {2012},
  month = {Jul},
  doi = {10.1021/ja3009314}
}



\end{document}